\documentclass[a4paper,10pt,twoside]{article}
\usepackage[utf8]{inputenc}
\usepackage{amsmath} 
\usepackage{amssymb}
\usepackage{amsfonts}
\usepackage[amsmath,thmmarks]{ntheorem}

\usepackage{pgf,tikz}
\usetikzlibrary{arrows}

\newtheorem{Th}{Theorem}
\newtheorem{lemma}[Th]{Lemma}
\newtheorem{cor}[Th]{Corollary}
\newtheorem{prop}[Th]{Proposition}

\theorembodyfont{\upshape}
\newtheorem{remark}[Th]{Remark}

\theoremstyle{nonumberplain}
\theoremsymbol{\ensuremath{\Box}}
\newtheorem{proof}{Proof}

\newcommand{\tauh}{\ensuremath{\mathcal{T}_H}}
\newcommand{\tauhtron}{\ensuremath{\overline{\mathcal{T}_H}}}
\newcommand{\taui}{\ensuremath{\mathcal{T}_I}}
\newcommand{\Ri}{\ensuremath{R_I}}
\newcommand{\Rhnoncom}{\ensuremath{\tilde{R}_H}}
\newcommand{\Rh}{\ensuremath{R_H}}
\newcommand{\projh}{\ensuremath{p_H}}
\newcommand{\proji}{\ensuremath{p_I}}
\newcommand{\trigi}{\ensuremath{\mathbf{\Delta^{3,I}}}}
\newcommand{\trigh}{\ensuremath{\mathbf{\Delta^{3,H}}}}
\newcommand{\recoli}{\ensuremath{\mathbf{\Phi^I}}}
\newcommand{\recolh}{\ensuremath{\mathbf{\Phi^H}}}

\newcommand{\vari}{\ensuremath{(\mathbf{\Delta^{3,I}}\setminus\mathbf{\Delta^{0,I}})/\mathbf{\Phi^I}}}

\newcommand{\nchain}[1]{\ensuremath{\Delta_{#1}^H}}
\newcommand{\bound}[1]{\ensuremath{\partial_{#1}^H}}

\newcommand{\trigdim}[1]{\ensuremath{\mathbf{\Delta^{#1}}}}
\newcommand{\trigdimH}[1]{\ensuremath{\mathbf{\Delta^{#1,H}}}}
\newcommand{\trigdimhtron}[1]{\ensuremath{\overline{\mathbf{\Delta^{#1,H}}}}}
\newcommand{\distE}{\ensuremath{K}}
\newcommand{\distF}{\ensuremath{F^K}}
\newcommand{\distT}{\ensuremath{T^K}}
\newcommand{\anyEH}{\ensuremath{A}}
\newcommand{\settaut}[1]{\ensuremath{\mathcal{F}}}

\newcommand{\hex}{\ensuremath{\mathcal{H}}}
\newcommand{\tri}{\ensuremath{\mathcal{T}}}
\newcommand{\signe}[1]{\ensuremath{\eta_{#1}}}
\newcommand{\signeb}[2]{\ensuremath{\eta_{#1}^{#2}}}

\newcommand{\Address}{{
  \bigskip
  \footnotesize
\textsc{Section de Mathématiques, Université de Genève, 2-4 rue du Lièvre, Case postale 64, 1211 Genève 4, Suisse.}\par\nopagebreak
  \textit{E-mail address}: \texttt{xavier.morvan@unige.ch}
}}

\title{A non-commutative generalisation of Thurston's gluing equations}
\author{Xavier Morvan\footnote{The work is supported by Swiss National Science Foundation, subsides SNF 200020\_149226.}}
\date{}

\begin{document}

\maketitle

\begin{abstract}
In his famous Princeton Notes, Thurston introduced the so-called gluing equations defining the deformation variety. Later,  Kashaev defined  a non-commutative ring from H-triangulations of 3-manifolds and observed that for trefoil and figure-eight knot complements the abelianization of this ring is isomorphic to the ring of regular functions on the deformation variety, \cite{Kashaev-definition_delta_groupoid, def_anneau, Kashaev-Delta-groupoids_and_ideal_triangulations}. In this paper, we prove that this is true for any knot complement in a homology sphere. We also analyse  some examples on other manifolds.
\end{abstract}

\section{Introduction}

In his famous Princeton notes \cite{Princeton_notes}, Thurston introduced the following gluing equations. 
Starting from an ideal triangulation $X$ of a cusped manifold, $M\setminus K$ assign to each edge of each tetraedron a shape parameter $z\in\mathbb{C}$ so that:
\begin{itemize}
	\item if $z_1, z_2, z_3$ are three shape parameters counter-clockwisely ordered around a vertex of a tetrahedron, then $z_1z_2z_3=-1$ and $z_2-z_1z_2=1$;
	 \item for each edge $E$ of $X$, let $z_1,\dots, z_n$ be the shape parameters of all the edges which are identified to $E$, then $\prod_{i=1}^n z_i=1$.
\end{itemize}

This defines a set of polynomial equations, hence an affine variety in $(\mathbb{C}\setminus\{0,1\})^{6n}$, where $n$ is the number of tetrahedra in the triangulation. This affine variety is called \emph{deformation variety}. We denote by \Ri\ its associated  ring of regular functions. The deformation variety may be empty for some triangulations. In that case, \Ri\ is defined to be the ring with one element.

Since then, the gluing equations have been extensively studied, see for example \cite{luo2012solving, solutions_for_gluing_eq, NEUMANN1985307, Petronio, Segerman_Generalisation, MR2866925}. 

Later, Kashaev introduced $\Delta$-groupoids \cite{Kashaev-definition_delta_groupoid} and $B'$-rings associated to ideal triangulations of manifolds and computed them for the trefoil and the figure eight knot complements \cite{Kashaev-Delta-groupoids_and_ideal_triangulations}.

Then, in \cite{def_anneau}, Kashaev introduced the ring \Rhnoncom, the abelianisation of which is studied in this article.
This ring is associated to a particular 1-vertex H-triangulation of a pair ($M^3,K$) (where $M^3$ is a connected, oriented, closed 3-manifold and $K$ is a knot) to which naturally corresponds  an ideal triangulation of the complement of $K$ in $M^3$. This will be precised in section \ref{Preliminaries}. In the case of the trefoil knot, \Rhnoncom\ is abelian and isomorphic to \Ri. In the case of the figure-eight knot, \Rhnoncom\ is not abelian, but its abelianization is isomorphic to \Ri.

We prove in the following that for any knot embedded in a homology 3-sphere there exists an isomorphism between the abelianization of the ring \Rhnoncom\ defined from particular H-triangulations and the ring of regular functions on the deformation variety of a corresponding ideal triangulation.
This isomorphism is explicitly constructed in section \ref{morphism} and the theorem proved in section \ref{proof}. In section \ref{examples} we give an example of a knot embedded in the 3-sphere and examples of knots embedded in other manifolds.

The author is sincerely grateful to Rinat Kashaev for very valuable discussions and usuful comments on this paper.

\section{Preliminaries}\label{Preliminaries}

\subsection{Triangulations}

Let $M$ be a connected, oriented, closed 3-manifold. A triangulation $\mathbf{\Delta}$ of $M$ is defined to consist of a pairwise disjoint union of oriented euclidian tetrahedra $\mathbf{\Delta^3}=\bigsqcup_{i=1}^n\Delta^{3}_i$, together with a collection $\mathbf{\Phi}$ of orientation-reversing affine isomorphisms pairing the faces of the tetrahedra in $\mathbf{\Delta^3}$, so that $M$ is homeomorphic to the identified space $\mathbf{\Delta^3} / \mathbf{\Phi}$. In the following, for $k\in\mathbb{N}$, $\trigdim{k}$ will denote the set of cells of dimension $k$ of the triangulation $\mathbf{\Delta}$.

An H-triangulation is a pair $(\mathbf{\Delta},K)$, where $\mathbf{\Delta}$ is a triangulation of a 3-manifold $M$ and $K$ is a hamiltonian subcomplex of the 1-skeleton of $\mathbf{\Delta}$. In the following, we will consider only \emph{particular} H-triangulations: we will assume that $\mathbf{\Delta}$ has only one vertex, and the hamiltonian subcomplex is given by a single edge \distE, which is contained in a single face, \distF\ of $\mathbf{\Delta}$, obtained from a single tetrahedron \distT\ of $\mathbf{\Delta}$, glued to itself as a closed book by $\mathbf{\Phi}$ along its two faces.
The edge \distE\ represents a knot embedded in $M$. Such a H-triangulation of $(M,\distE)$ will be denoted \tauh.  In order to stress the "H-nature" of \tauh, the set of $k$ dimensional cells of \tauh\ will be denoted \trigdimH{k}\ and the set of face pairings, \recolh. We will denote $\projh:\trigh\ \rightarrow\ \trigh/ \recolh$ the identification projection. \distE, \distF, \distT\ will respectively be called \emph{distinguished edge}, \emph{distinguished face} and \emph{distinguished tetrahedron} (see figure \ref{example of distinguished tetrahedron}).

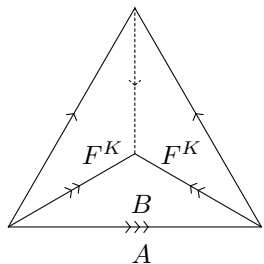
\begin{figure}
\begin{center}
\begin{tikzpicture}[line cap=round,line join=round,>=triangle 45,x=0.3699783241293612cm,y=0.37354724687015906cm]
\clip(-1.22,-7.66) rectangle (8.24,1.71);
\draw [dash pattern=on 1pt off 1pt] (3.5,1.58)-- (3.5,-3.61);
\draw [dash pattern=on 1pt off 1pt] (3.5,-1.17) -- (3.3,-1.01);
\draw [dash pattern=on 1pt off 1pt] (3.5,-1.17) -- (3.7,-1.01);
\draw (8,-6.21)-- (3.5,-3.61);
\draw (5.48,-4.75) -- (5.71,-4.66);
\draw (5.48,-4.75) -- (5.51,-5.01);
\draw (5.75,-4.91) -- (5.99,-4.81);
\draw (5.75,-4.91) -- (5.79,-5.17);
\draw (-1,-6.21)-- (3.5,-3.61);
\draw (1.52,-4.75) -- (1.49,-5.01);
\draw (1.52,-4.75) -- (1.29,-4.66);
\draw (1.25,-4.91) -- (1.21,-5.17);
\draw (1.25,-4.91) -- (1.01,-4.81);
\draw (-1,-6.21)-- (8,-6.21);
\draw (3.66,-6.21) -- (3.5,-6.41);
\draw (3.66,-6.21) -- (3.5,-6.01);
\draw (3.34,-6.21) -- (3.18,-6.41);
\draw (3.34,-6.21) -- (3.18,-6.01);
\draw (3.97,-6.21) -- (3.82,-6.41);
\draw (3.97,-6.21) -- (3.82,-6.01);
\draw (3.05,-6.51) node[anchor=north west] {$A$};
\draw (3.02,-4.73) node[anchor=north west] {$B$};
\draw (4.08,-2.84) node[anchor=north west] {$\distF$};
\draw (1.28,-2.84) node[anchor=north west] {$\distF$};
\draw (-1,-6.21)-- (3.5,1.56);
\draw (1.33,-2.19) -- (1.43,-2.43);
\draw (1.33,-2.19) -- (1.07,-2.22);
\draw (8,-6.21)-- (3.5,1.56);
\draw (5.67,-2.19) -- (5.93,-2.22);
\draw (5.67,-2.19) -- (5.57,-2.43);
\end{tikzpicture}
\caption{Example of a distinguished tetrahedron \distT. Its distinguished edge \distE\ is dotted. \label{example of distinguished tetrahedron}}
\end{center}
\end{figure}

\begin{prop}
	For any knot $K$ in $\mathbb{S}^3$, there exists a particular H-triangulation of $(\mathbb{S}^3,K)$.
\end{prop}
\begin{proof}
There exist algorithms for one vertex H-triangulations of couples $(\mathbb{S}^3,K)$. See, for example, \cite{Kashaev2015}, section 4. A slight modification of this algorithm allows one to get a particular H-triangulation: at the last stage of the decomposition of the 3-cell onto several tetrahedra, one has to extract a distinguished tetrahedron, and the previous steps of the algorithm always make this possible.
\end{proof}

From \tauh, one gets a cell decomposition of another manifold by removing an open neighbourhood of the vertex. This is equivalent to truncating the tetrahedra, so that the new 3-cells are bounded by \emph{triangular} and \emph{hexagonal} faces. The triangular faces are bounded by \emph{short edges}, and hexagonal faces are bounded by short and \emph{long edges} which are remnant of the edges of \tauh. The obtained manifold is the complement in $M$ of an open ball. Such a cell decomposition will be denoted \tauhtron\ and the set of $k$ dimensional cells in \tauhtron\ will be denoted $\trigdimhtron{k}$. Figure \ref{truncated distinguished tetrahedron} is an example of a truncated tetrahedron.
We will denote $\hex\subset \trigdimhtron{2} / \recolh$ and $\mathcal{\tilde{H}}\subset \trigdimhtron{2}$ the sets of non distinguished hexagonal faces, $\tri$ the set of triangular faces disjoint from the distinguished edge, $\mathcal{L}\subset \trigdimhtron{1} / \recolh$ and $\mathcal{\tilde{L}}\subset \trigdimhtron{1}$ the sets of non distinguished long edges, $\mathcal{S}\subset \trigdimhtron{1} / \recolh$, and $\mathcal{\tilde{S}}\subset \trigdimhtron{1}$ the sets of short edges disjoint from the knot.

An ideal triangulation of a connected, oriented, cusped manifold $N$, is defined to consist of a pairwise disjoint union of oriented euclidian tetrahedra $\trigi=\bigsqcup_{i=1}^n\Delta^{3,I}_i$, together with a set $\recoli$ of orientation-reversing affine isomorphisms pairing the faces of the tetrahedra in $\trigi$ so that $N=\vari$. Such a triangulation will be denoted \taui. We will denote $\proji: \trigi\setminus \mathbf{\Delta^{0,I}} \rightarrow \vari$ the identification projection.
In the following, we will consider a knot $K$ embedded in a 3-manifold $M$ and ideal triangulations of $N=M\setminus K$, such that $\trigi / \recoli$ is a pseudomanifold having one singular point, its only vertex, corresponding to $K$. The tetrahedra of \taui\ can be seen as hyperbolic ideal tetrahedra with vertices at infinity.

Let us remark that from a particular H-triangulation of a pair $(M,K)$, one can get an ideal triangulation of $M\setminus K$ by collapsing the distinguished edge to a point in such a way that the distinguished face \distF\ is collapsed to an edge while the two other faces of the distinguished tetrahedron (not bounded by distinguished edge) are identified with each other.
For example, in figure \ref{example of distinguished tetrahedron}, the edges with simple and double arrows are identified with each other; face $\distF$ collapses and faces $B$ and $A$ are identified with each other.
Then, each cell of \tauh\ different from the distinguished ones have a canonical corresponding cell in \taui.

One can find more information on triangulations of 3-manifolds in \cite{BenPetr, livre_Matveev}.

From an ideal triangulation of a knot complement in a 3-manifold, one can construct a ring \Ri\ as in the introduction.

\subsection{Definition of \Rh} \label{section def de Rh}

According to \cite{def_anneau}, the ring \Rhnoncom\ is defined from $\tauhtron$ with oriented short edges by the following presentation. The set of generators is given by associating to each oriented short edge $e\in\mathcal{S}$, disjoint with the distinguished edge, a pair of generators $(u_e,v_e)$; and the set of relations:
\begin{itemize}
	\item if $\bar{e}$ is the edge $e$ with opposite orientation, $u_{\bar{e}}=u_e^{-1}$ and $v_{\bar{e}}=-u_e^{-1}v_e$;
	\item \emph{hexagonal face relation}: if $\check{e}$ is the unique oriented short edge such that it belongs to the same hexagonal face as $e$, and the terminal points of $e$ and $\check{e}$ form the boundary of a long edge, then $u_{\check{e}}=v_e$ and $v_{\check{e}}=u_e$;
	\item \emph{triangular face relations}: if $e_1, e_2, e_3$ are cyclically oriented short edges constituting the boundary of a triangular face, then $u_{e_1}u_{e_2}u_{e_3}=1$ and $u_{e_1}u_{e_2}v_{e_3}+u_{e_1}v_{e_2}+v_{e_1}=0$.
\end{itemize}
One can easily check that there exists a representation of the knot  group onto $GL(2,\Rhnoncom)$ by associating the matrix 
$\begin{pmatrix}
	u_e&v_e\\0&1
\end{pmatrix}$ to each oriented short edge disjoint from the distinguished edge and the matrix
$\begin{pmatrix}
	0&1\\1&0
\end{pmatrix}$ to each (oriented) long edge different from the distinguished one.

Note that the ring \Rhnoncom\ is not necessarily commutative. We will denote by \Rh\ the abelianization of \Rhnoncom.

\begin{lemma}\label{relations for the distinguished tetrahedra}
	Let the truncated distinguished tetrahedron be labelled as in figure \ref{truncated distinguished tetrahedron}, then $u_l=u_m$, $u_p=v_m^{-1}v_l$, $v_p=0$.
\end{lemma}
\begin{proof}
	Let us write the two triangular face relations where all generators with indices $i_1, i_2, j_1, j_2$ are expressed in terms of $(u_l,v_l)$ and $(u_m,v_m)$ through the use of the hexagonal faces relations.

For the left hand triangle:
\begin{align}
	1&=v_mu_pv_l^{-1} \label{eq1} \\
	0&=-v_mu_pv_l^{-1}u_l+v_mv_p+u_m \label{eq2}
\end{align}
and the right hand triangle:
\begin{align}
	1&=u_m^{-1}v_mu_pv_l^{-1}u_l \label{eq3} \\ 
	0&=-u_m^{-1}v_mu_p(v_l^{-1}u_l)u_l^{-1}+u_m^{-1}v_mv_p+u_m^{-1} \label{eq4}
\end{align}
Equations (\ref{eq1}) and (\ref{eq3}) imply that $u_m=u_l$. Then (\ref{eq2}) implies that $v_p=0$ because $v_m$ is invertible.
\end{proof}

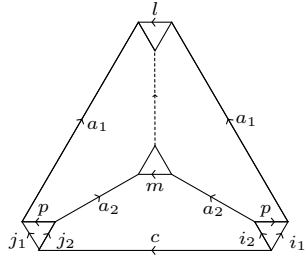
\begin{figure}
\begin{center}
	\begin{tikzpicture}[line cap=round,line join=round,>=triangle 45,x=0.4372027811400509cm,y=0.43673740521645454cm]
\clip(-1.24,-6.4) rectangle (8.2,1.61);
\draw (3.5,0.06)-- (3,0.93);
\draw (0,-6)-- (0.5,-5.13);
\draw (0.3,-5.49) -- (0.35,-5.63);
\draw (0.3,-5.49) -- (0.15,-5.51);
\draw (0.5,-5.13)-- (-0.5,-5.13);
\draw (-0.09,-5.13) -- (0,-5.02);
\draw (-0.09,-5.13) -- (0,-5.25);
\draw (7,-6)-- (6.5,-5.13);
\draw (6.7,-5.49) -- (6.85,-5.51);
\draw (6.7,-5.49) -- (6.65,-5.63);
\draw (6.5,-5.13)-- (7.5,-5.13);
\draw (7.09,-5.13) -- (7,-5.25);
\draw (7.09,-5.13) -- (7,-5.02);
\draw (3.5,-2.82)-- (4,-3.69);
\draw (3,0.93)-- (3.5,0.06);
\draw (3.5,-2.82)-- (3,-3.69);
\draw (0.5,-5.13)-- (-0.5,-5.13);
\draw (-0.5,-5.13)-- (3,0.93);
\draw (1.3,-2.02) -- (1.35,-2.16);
\draw (1.3,-2.02) -- (1.15,-2.04);
\draw (3.5,0.06)-- (4,0.93);
\draw (4,0.93)-- (7.5,-5.13);
\draw (7.5,-5.13)-- (6.5,-5.13);
\draw (6.5,-5.13)-- (4,-3.69);
\draw (5.17,-4.37) -- (5.31,-4.31);
\draw (5.17,-4.37) -- (5.19,-4.52);
\draw (7,-6)-- (6.5,-5.13);
\draw (4,-3.69)-- (3,-3.69);
\draw (3.41,-3.69) -- (3.5,-3.57);
\draw (3.41,-3.69) -- (3.5,-3.81);
\draw (0.5,-5.13)-- (0,-6);
\draw (7,-6)-- (7.5,-5.13);
\draw (7.3,-5.49) -- (7.35,-5.63);
\draw (7.3,-5.49) -- (7.15,-5.51);
\draw (7.5,-5.13)-- (4,0.93);
\draw (5.7,-2.02) -- (5.85,-2.04);
\draw (5.7,-2.02) -- (5.65,-2.16);
\draw (4,0.93)-- (3,0.93);
\draw (3.41,0.93) -- (3.5,1.05);
\draw (3.41,0.93) -- (3.5,0.81);
\draw (3,0.93)-- (-0.5,-5.13);
\draw (0.5,-5.13)-- (3,-3.69);
\draw (1.83,-4.37) -- (1.81,-4.52);
\draw (1.83,-4.37) -- (1.69,-4.31);
\draw (7,-6)-- (0,-6);
\draw (3.41,-6) -- (3.5,-5.88);
\draw (3.41,-6) -- (3.5,-6.12);
\draw [dash pattern=on 1pt off 1pt] (3.5,-2.82)-- (3.49,0.08);
\draw [dash pattern=on 1pt off 1pt] (3.5,-1.28) -- (3.61,-1.37);
\draw [dash pattern=on 1pt off 1pt] (3.5,-1.28) -- (3.38,-1.37);
\draw (0.02,-5.97)-- (-0.5,-5.13);
\draw (-0.29,-5.48) -- (-0.14,-5.49);
\draw (-0.29,-5.48) -- (-0.34,-5.62);
\begin{scriptsize}
\draw[color=black] (0.84,-5.71) node {$j_2$};
\draw[color=black] (0.14,-4.76) node {$p$};
\draw[color=black] (6.29,-5.65) node {$i_2$};
\draw[color=black] (6.83,-4.8) node {$p$};
\draw[color=black] (1.78,-2.23) node {$a_1$};
\draw[color=black] (5.26,-4.8) node {$a_2$};
\draw[color=black] (3.5,-4.06) node {$m$};
\draw[color=black] (7.81,-5.77) node {$i_1$};
\draw[color=black] (6.25,-2.06) node {$a_1$};
\draw[color=black] (3.51,1.35) node {$l$};
\draw[color=black] (2.12,-4.73) node {$a_2$};
\draw[color=black] (3.5,-5.62) node {$c$};
\draw[color=black] (-0.61,-5.71) node {$j_1$};
\end{scriptsize}
\end{tikzpicture}
\caption{Truncated distinguished tetrahedron \label{truncated distinguished tetrahedron}}
\end{center}
\end{figure}

\section{A ring homomorphism from \Ri\ to \Rh} \label{morphism}

We define a ring homomorphism $f:\Ri\rightarrow\Rh$ as follows.
Let $E$ be an edge of a tetrahedron of the ideal triangulation with shape parameter $z$. Let us still call $E$ the corresponding long edge in $\mathcal{\tilde{L}}$. We choose a boundary point of $E$ in \trigdimhtron{0}, and let $e_i$ and $e_j$ be the short edges in $\mathcal{\tilde{S}}$ sharing this point at their origins and such that $E,e_i,e_j$ are clockwisely ordered.
Let $(u_i,v_i)$ and $(u_j,v_j)$ be the couples of generators in \Rh\ assigned to $e_i$ and $e_j$. We define, $f(z)=v_iv_j^{-1}$ (see figure \ref{definition of f}).

\begin{remark}\label{autre def de f}
	Keeping the notation of figure \ref{definition of f}, $f$ can be equivalently defined as $f(z)=u_pu_q^{-1}$.
\end{remark}

\begin{figure}
	\begin{center}
		\begin{tikzpicture}[line cap=round,line join=round,>=triangle 45,x=0.3961178933722044cm,y=0.4219697979922829cm]
\clip(-0.87,-6.75) rectangle (7.97,1.54);
\draw (3.5,0.06)-- (3,0.93);
\draw (3.2,0.57) -- (3.35,0.55);
\draw (3.2,0.57) -- (3.15,0.44);
\draw (-0.5,-5.13)-- (0,-6);
\draw (0,-6)-- (0.5,-5.13);
\draw (0.5,-5.13)-- (-0.5,-5.13);
\draw (-0.09,-5.13) -- (0,-5.02);
\draw (-0.09,-5.13) -- (0,-5.25);
\draw (7.5,-5.13)-- (7,-6);
\draw (7,-6)-- (6.5,-5.13);
\draw (6.5,-5.13)-- (7.5,-5.13);
\draw (7.09,-5.13) -- (7,-5.25);
\draw (7.09,-5.13) -- (7,-5.02);
\draw (3.5,-2.82)-- (4,-3.69);
\draw (3.8,-3.34) -- (3.65,-3.32);
\draw (3.8,-3.34) -- (3.85,-3.2);
\draw (3,0.93)-- (3.5,0.06);
\draw (3.5,0.06)-- (3.5,-2.82);
\draw (3.5,-2.82)-- (3,-3.69);
\draw (3.2,-3.34) -- (3.15,-3.2);
\draw (3.2,-3.34) -- (3.35,-3.32);
\draw (3,-3.69)-- (0.5,-5.13);
\draw (0.5,-5.13)-- (-0.5,-5.13);
\draw (-0.5,-5.13)-- (3,0.93);
\draw (3.5,0.06)-- (4,0.93);
\draw (3.8,0.57) -- (3.85,0.44);
\draw (3.8,0.57) -- (3.65,0.55);
\draw (4,0.93)-- (7.5,-5.13);
\draw (7.5,-5.13)-- (6.5,-5.13);
\draw (6.5,-5.13)-- (4,-3.69);
\draw (3.5,-2.82)-- (3.5,0.06);
\draw (7,-6)-- (6.5,-5.13);
\draw (6.5,-5.13)-- (4,-3.69);
\draw (4,-3.69)-- (3,-3.69);
\draw (3.41,-3.69) -- (3.5,-3.57);
\draw (3.41,-3.69) -- (3.5,-3.81);
\draw (3,-3.69)-- (0.5,-5.13);
\draw (0.5,-5.13)-- (0,-6);
\draw (0,-6)-- (7,-6);
\draw (7,-6)-- (7.5,-5.13);
\draw (7.5,-5.13)-- (4,0.93);
\draw (4,0.93)-- (3,0.93);
\draw (3.41,0.93) -- (3.5,1.05);
\draw (3.41,0.93) -- (3.5,0.81);
\draw (3,0.93)-- (-0.5,-5.13);
\draw (-0.5,-5.13)-- (0,-6);
\draw (0,-6)-- (7,-6);
\begin{scriptsize}
\draw[color=black] (2.95,0.21) node {i};
\draw[color=black] (0.14,-4.76) node {p};
\draw[color=black] (6.83,-4.8) node {q};
\draw[color=black] (4.28,-3.09) node {n};
\draw[color=black] (3.13,-1.23) node {z};
\draw[color=black] (2.88,-3.1) node {k};
\draw[color=black] (4.04,0.21) node {j};
\draw[color=black] (3.5,-4.06) node {m};
\draw[color=black] (3.51,1.35) node {l};
\end{scriptsize}
\end{tikzpicture}
\caption{$f(z)=v_iv_j^{-1}$ \label{definition of f}}
	\end{center}
\end{figure}
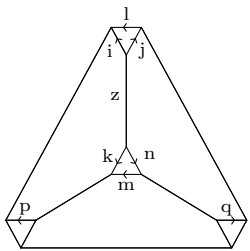

Let us check that $f$ is a well defined ring homomorphism.

	First, the definition of $f$ does not depend on the choice of the boundary point of E: if $e_k=\overline{\check{\bar{e_i}}}$ and $e_n=\overline{\check{\bar{e_j}}}$, then $v_i=v_k^{-1}$ and $v_j=v_n^{-1}$ so that $v_iv_j^{-1}=v_nv_k^{-1}$.

Let $E_1, E_2, E_3$ be three long edges counter-clockwisely ordered around a vertex of a tetrahedron in \trigi, with shape parameters $z_1,z_2,z_3$, and let $e_i, e_j, e_k$ be three short edges bounding a triangular face obtained after truncation along the corresponding vertex and such that $E_1, e_i, e_j$ and $E_2, \overline{e_j}, e_k$ are clockwisely ordered. Then, the triangular face relations are $u_ju_ku_i^{-1}=1$ and $-v_i+u_j v_k+v_j=0$.

	So, 
\begin{equation*}
	f(z_1)f(z_2)f(z_3)=(v_iv_j^{-1})(-u_j^{-1}v_jv_k^{-1})(u_k^{-1}v_kv_i^{-1}u_i)=-u_j^{-1}u_k^{-1}u_i=-1
\end{equation*}

	and 
\begin{equation*}
	f(z_2)-f(z_1)f(z_2)=-u_j^{-1}v_jv_k^{-1}+v_iv_j^{-1}u_j^{-1}v_jv_k^{-1}=u_j^{-1}(-v_j+v_i)v_k^{-1}=1
\end{equation*}
	
Let us now check that the gluing equations are respected.
If the ideal triangulation of the knot complement contains an edge with only one preimage by \proji, then both \Ri\ and \Rh\ are the ring with one element and all conditions are trivially satisfied.

Before considering the general case, let us first remark the following. Let two triangles in $\tri$ be glued together along an oriented short edge $e_j$ with initial point corresponding to long edges with shape parameters $z_1$ and $z_2$.
Let $e_i,e_j,e_k$ be short edges clockwisely ordered as in figure \ref{glued triangles}.
	Then,	$f(z_1)f(z_2)=(v_iv_j^{-1})(v_jv_k^{-1})=v_iv_k^{-1}$.

	Let us now consider separately few different cases. First, if an edge $E$ in the ideal triangulation has no corresponding one in a distinguished tetrahedron of the H-triangulation, then the product of the $f(z_i)$ around $E$ telescopically reduces to 1. 

Second, $E$ has a corresponding edge opposite to the distinguished edge in the distinguished tetrahedron. Then, after telescopication, keeping notations of lemma \ref{relations for the distinguished tetrahedra}, the product of $f(z_i)$'s around $E$ reduces to $v_{\check{m}}v_{\check{l}}^{-1}$ which is equal to $1$ due to lemma \ref{relations for the distinguished tetrahedra}.

Third case, $E$ has a corresponding edge in the distinguished tetrahedron in the boundary of the distinguished face. Again, and after telescopication, the product of $f(z_i)$'s around $E$ reduces to a product of parameters $v_i$ which is equal to 1. In order to see this, it is easier to use the equivalent definition of $f$ : $f(z)=u_pu_q^{-1}$ as in remark \ref{autre def de f}. Then, by using the notation of figure 2, the edge $a_2$ will contribute $u_{i_2}u_{j_2}^{-1}$ and the edge $a_1$ will contribute $u_{i_1}^{-1}u_{j_1}$ (both $a_k$ appear twice in the distinguished tetrahedron). The product of those two contributions is 1 due to hexagonal face relations of figure \ref{truncated distinguished tetrahedron}.

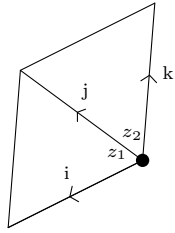
\begin{figure}
	\begin{center}
\begin{tikzpicture}[line cap=round,line join=round,>=triangle 45,x=0.7653061224489793cm,y=0.8177570093457945cm]
\clip(-2.12,-1.42) rectangle (1.8,2.86);
\draw [line width=0.4pt] (0.86,0)-- (-1.24,1.46);
\draw [line width=0.4pt] (-0.26,0.78) -- (-0.12,0.83);
\draw [line width=0.4pt] (-0.26,0.78) -- (-0.26,0.63);
\draw [line width=0.4pt] (-1.24,1.46)-- (-1.45,-1.09);
\draw [line width=0.4pt] (-1.45,-1.09)-- (0.86,0);
\draw (0.86,0)-- (-1.45,-1.09);
\draw (-0.39,-0.59) -- (-0.35,-0.42);
\draw (-0.39,-0.59) -- (-0.24,-0.67);
\draw (0.86,0)-- (1.07,2.55);
\draw (0.98,1.38) -- (1.1,1.26);
\draw (0.98,1.38) -- (0.83,1.29);
\draw (1.07,2.55)-- (-1.24,1.46);
\begin{scriptsize}
\draw [color=black] (0.86,0) circle (1.5pt);
\draw[color=black] (0.42,0.1) node {$z_1$};
\draw[color=black] (-0.12,1.12) node {j};
\draw[color=black] (-0.44,-0.2) node {i};
\fill [color=black] (0.86,0) circle (2.5pt);
\draw[color=black] (0.68,0.42) node {$z_2$};
\draw[color=black] (1.3,1.42) node {k};
\end{scriptsize}
\end{tikzpicture}
\caption{Glued triangles \label{glued triangles}}
	\end{center} 

\end{figure}

\begin{Th}\label{theorem}
	If $H_1(M,\mathbb{Z})=0$, then $f$ is a ring isomorphism.
\end{Th}

\begin{remark}
	According to Poincaré duality and Universal coefficients theorem, if $H_1(M,\mathbb{Z})=0$, then $H_2(M,\mathbb{Z})=0$.
\end{remark}

\section{Proof of theorem \ref{theorem}} \label{proof}

First of all, let us clear up the particular cases where the ideal triangulation of a knot complement contains an edge with only one preimage by $\proji$. Such an edge must be in a tetrahedron closed as a book around it. Then, the shape parameter of that edge must be equal to $1$, which means that the deformation variety is empty. As a consequence, the ring $\Ri$ is the ring with one element. Besides, according to lemma \ref{relations for the distinguished tetrahedra}, in that case, there is an invertible element in \Rh\ equal to 0. So \Rh\ is also the ring with one element and the isomorphism between \Ri\ and \Rh\ is trivial.

For the general case, let us first remark that the image of an element of \Ri\ by $f$ is a priori a product of an even number of elements in \Rh. Then, in order to prove the surjectivity of $f$, the main difficulty is to write an equality in \Rh\ with one generating element on the left hand side and an even number of generating elements on the right hand side. For that, we will use the triviality of $H_1(M,\mathbb{Z})$.

From now on, we will say that an oriented short edge $e\in\mathcal{\tilde{S}}$ is \emph{parallel} (resp. \emph{antiparallel}) to an oriented long edge $E$ of $\mathcal{\tilde{L}}$ if there exists $E'$ of $\projh^{-1}(\projh(E))$ such that $e$ and $E'$ are disjointly contained in the boundary of one and the same hexagonal face $F$, and if the orientations of $\bar{e}$ (resp. $e$) and $E'$ are induced from the orientation of $F$.
We will also say that two oriented short edges $e$ and $e'$ are \emph{parallel} (resp. \emph{antiparallel}) if there exists a long edge $E$ of $\mathcal{\tilde{L}}$ such that $e$ and $e'$ are both parallel or antiparallel (resp. one is parallel and another one is antiparallel) to $E$.
We will also say that an edge $e\in\mathcal{S}$ is parallel (resp. antiparallel) to an edge $E\in\mathcal{L}$ if there exists a parallel (resp. antiparallel) preimage of $e$ and $E$ by $\projh$.
For instance, in figure \ref{truncated distinguished tetrahedron}, $m$, $c$ and $l$ are parallel.
\begin{lemma}\label{Lemme}
	If $e$ and $e'$ are parallel short edges in $\mathcal{S}$, then there exists $m\in\Ri$ such that $f(m)=u_eu_{e'}^{-1}$.
\end{lemma}
\begin{proof}
	We will prove that there exists $m\in\Ri$ such that $f(m)=v_{\check{e}}v_{\check{e'}}^{-1}$(see subsection \ref{section def de Rh} for the definition of $\check{e}$ and $\check{e'}$), which is equivalent to the lemma \ref{Lemme}. 

Let $E\in\mathcal{L}$ be parallel to $e$ and $e'$. Then, among the triangles containing the edge $E$ in \tauhtron\ there are some which also contain the edge $\check{e}$ or $\check{e'}$.
On the disk composed by the triangular faces containing a given end point of $E$, choose a path from $\check{e}$ to $\check{e'}$. Let $e_1,\ldots,e_n$ be short edges met by this path and oriented so that they all have the same starting point (here $\check{e}=e_1$ and $\check{e'}=e_n$). Let $\tilde{E}_1,\ldots,\tilde{E}_{n-1}\in\projh^{-1}(E)$ be long edges intersecting $e_1,\ldots,e_n$ in their starting point,  and let $z_1,\ldots,z_{n-1}$ be the corresponding parameters in the ideal triangulation (see figure \ref{image pour lemme}). Let $m=\prod_{i=1}^{n-1} z_i\in\Ri$.
According to the gluing equations in $\Ri$, $m$ does not depend on the choice of the path from $\check{e}$ to $\check{e'}$. Besides, $\prod_{i=1}^{n-1} f(z_i)=\prod_{i=1}^{n-1} v_iv_{i+1}^{-1}=v_1v_n^{-1}=v_{\check{e}}v_{\check{e'}}^{-1}$ (here, $(u_i,v_i)$ is the couple of generators 
in $\Rh$ associated to the short edge $e_i$).
\end{proof}

It will be usefull in the following to denote $m_{e\rightarrow e'}$ for such a word.

\begin{figure}
	\begin{center}
\begin{tikzpicture}[line cap=round,line join=round,>=triangle 45,x=0.4892367906066535cm,y=0.4726890756302519cm]
\clip(-1.7,-4.1) rectangle (8.52,5.42);
\draw [shift={(3.31,0.63)}] (0,0) -- (178.01:1.8) arc (178.01:214.01:1.8) -- cycle;
\draw [shift={(3.31,0.63)}] (0,0) -- (142.01:1.8) arc (142.01:178.01:1.8) -- cycle;
\draw [shift={(3.31,0.63)},dash pattern=on 1pt off 1pt] (0,0) -- (-1.99:1.8) arc (-1.99:142.01:1.8) -- cycle;
\draw [shift={(3.31,0.63)}] (0,0) -- (-37.99:1.8) arc (-37.99:-1.99:1.8) -- cycle;
\draw (3.31,0.63)-- (-0.26,-1.78);
\draw (1.44,-0.63) -- (1.45,-0.46);
\draw (1.44,-0.63) -- (1.6,-0.69);
\draw (3.31,0.63)-- (-0.99,0.78);
\draw (1.05,0.71) -- (1.16,0.84);
\draw (1.05,0.71) -- (1.15,0.57);
\draw (3.31,0.63)-- (-0.08,3.29);
\draw (1.53,2.02) -- (1.7,2.06);
\draw (1.53,2.02) -- (1.53,1.85);
\draw (3.31,0.63)-- (6.71,-2.02);
\draw (5.1,-0.76) -- (4.93,-0.8);
\draw (5.1,-0.76) -- (5.1,-0.59);
\draw (3.31,0.63)-- (7.62,0.48);
\draw (5.57,0.55) -- (5.46,0.42);
\draw (5.57,0.55) -- (5.47,0.69);
\draw [shift={(3.31,0.63)},dash pattern=on 2pt off 2pt]  plot[domain=-0.03:2.48,variable=\t]({1*4.31*cos(\t r)+0*4.31*sin(\t r)},{0*4.31*cos(\t r)+1*4.31*sin(\t r)});
\draw [shift={(3.31,0.63)},-<,dash pattern=on 1pt off 1pt] (-1.99:1.8) arc (-1.99:142.01:1.8);
\draw [shift={(3.31,0.63)},dash pattern=on 2pt off 2pt]  plot[domain=3.74:5.62,variable=\t]({1*4.31*cos(\t r)+0*4.31*sin(\t r)},{0*4.31*cos(\t r)+1*4.31*sin(\t r)});
\draw (-0.26,-1.78)-- (-0.99,0.78);
\draw (-0.99,0.78)-- (-0.08,3.29);
\draw (7.62,0.48)-- (6.71,-2.02);
\begin{scriptsize}
\draw[color=black] (1.48,-1.26) node {$e_1$};
\draw[color=black] (0.9,0.28) node {$e_2$};
\draw[color=black] (1.28,1.74) node {$e_3$};
\draw[color=black] (5.44,-1.34) node {$e_n$};
\draw[color=black] (6.34,0.16) node {$e_{n-1}$};
\draw[color=black] (2.58,0.3) node {$z_1$};
\draw[color=black] (2.58,0.86) node {$z_2$};
\draw[color=black] (4.28,0.28) node {$z_{n-1}$};
\end{scriptsize}
\end{tikzpicture}
\caption{$\prod_{i=1}^{n-1}f(z_i)=v_1v_n^{-1}$\label{image pour lemme}}
	\end{center}
\end{figure}
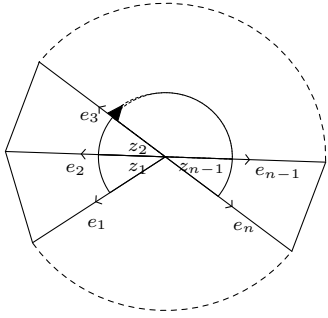

\begin{cor}\label{corollary}
	Let $e\in\mathcal{S}$ and $m\in\Ri$ be such that $f(m)=u_e$. Then, for any short edge $e'$ parallel to $e$, there exists $m'\in\Ri$ such that $f(m')=u_{e'}$.
\end{cor}

Being a cell complex decomposition of $M$, \tauh\ induces a presentation of the first homology group with 1-cells as generators, and 2-cells as relations. Let us denote by \distE\ the distinguished edge corresponding to the knot in $M$.
Let us also denote by $E_i^H$ for $1\leqslant i \leqslant n$ all the other edges in $\trigdimH{1} / \recolh$ and $E_i^{j,H}$ the edges of  $\trigdimH{1}$ such that $\projh(E_i^{j,H})=E_i^H$.
Let us denote by $\nchain{l}(M)$ the group of oriented $l$-chains (the free abelian group with basis the oriented open $l$-simplices), with boundary maps: $\bound{l}:\nchain{l}(M)\rightarrow\nchain{l-1}(M)$.
Let us denote by \distF\ the distinguished face of $\trigdimH{2} / \recolh$, and $F_i^H$ for $1\leqslant i \leqslant m$ the other faces. Then $\bound{2}(\distF)=\epsilon_{\distE}\distE+\epsilon_j E_j^H+\epsilon_{j'} E_{j'}^H$ ($\epsilon_i=\pm 1$).
Besides, according to the particular H-triangulations considered here, $\distF$ is the only face of \trigdimH{2} / \recolh\ whose boundary contains $\distE$. Thus, one can consider another presentation of the first homology group with generators the edges in $\trigdimH{1} / \recolh$ without the distinguished one, and relations corresponding to non-distinguished faces. On what follows, we assume that $K$ and \distF\ are removed from the corresponding sets of cells.

Let us assume from now on that $H_1(M,\mathbb{Z})=0$ and define some useful notations.
Let \anyEH\ be an edge in $\trigdimH{1} / \recolh$. Then, there exists a map $\signe{\anyEH}: \trigdimH{2} / \recolh \rightarrow \mathbb{Z}$ such that $$\anyEH=\sum_{F\in\trigdimH{2}/ \recolh}\signe{\anyEH}(F)\bound{2}(F)\in \nchain{1}(M).$$ It is a tautological equality in a free abelian group.

For any $F\in\mathcal{\tilde{H}}$ let
$$\mathcal{\tilde{E}}_F=\{E\in\mathcal{\tilde{L}} \mid E \text{ is on the boundary of }F\}$$
and let us define a set 
$$X=\{(F,E) \mid F\in\mathcal{\tilde{H}}, E \in\mathcal{\tilde{E}}_F \} $$
and a map $r:X\rightarrow\mathcal{\tilde{S}}$ in the following way. $E$ is on the boundary of two faces in $\mathcal{\tilde{H}}$, $F$ and $F'$, so $r(F,E)$ is the short edge on the boundary of $F'$ disjoint from $E$ and parallel to it (see figure \ref{def de r}).

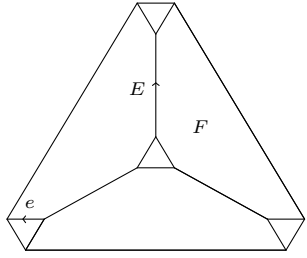
\begin{figure}
	\begin{center}
\begin{tikzpicture}[line cap=round,line join=round,>=triangle 45,x=0.4892367906066535cm,y=0.4726890756302519cm]
\clip(-0.98,-6.56) rectangle (8.11,1.41);
\draw (3.5,0.06)-- (3,0.93);
\draw (0,-6)-- (0.5,-5.13);
\draw (0.5,-5.13)-- (-0.5,-5.13);
\draw (-0.07,-5.13) -- (0,-5.04);
\draw (-0.07,-5.13) -- (0,-5.23);
\draw (3.5,-2.82)-- (3,-3.69);
\draw (3.5,0.06)-- (4,0.93);
\draw (4,0.93)-- (7.5,-5.13);
\draw (7.5,-5.13)-- (6.5,-5.13);
\draw (6.5,-5.13)-- (4,-3.69);
\draw (4,-3.69)-- (3.5,-2.82);
\draw (7,-6)-- (6.5,-5.13);
\draw (6.5,-5.13)-- (4,-3.69);
\draw (4,-3.69)-- (3,-3.69);
\draw (3,-3.69)-- (0.5,-5.13);
\draw (0.5,-5.13)-- (0,-6);
\draw (0,-6)-- (7,-6);
\draw (7,-6)-- (7.5,-5.13);
\draw (7.5,-5.13)-- (4,0.93);
\draw (4,0.93)-- (3,0.93);
\draw (3,0.93)-- (-0.5,-5.13);
\draw (-0.5,-5.13)-- (0,-6);
\draw (0,-6)-- (7,-6);
\draw (3.5,-2.82)-- (3.5,0.06);
\draw (3.5,-1.31) -- (3.59,-1.38);
\draw (3.5,-1.31) -- (3.41,-1.38);
\begin{scriptsize}
\draw[color=black] (0.12,-4.7) node {$e$};
\draw[color=black] (4.7,-2.53) node {$F$};
\draw[color=black] (3,-1.48) node {$E$};
\end{scriptsize}
\end{tikzpicture}
\caption{Definition of $r: e=r(F,E)$}\label{def de r}
\end{center}
\end{figure}

We fix a section $\sigma: \hex \rightarrow \mathcal{\tilde{H}}$ with the image disjoint with the distinguished tetraedron. Denote $\tilde{\settaut{\anyEH}}=\sigma(\hex)$ and reinterpret $\signe{\anyEH}$ as a map from $\tilde{\settaut{\anyEH}}$ to $\mathbb{Z}$.

We define 

$$U_{\anyEH}=\prod_{F\in\tilde{\settaut{\anyEH}}}\prod_{E\in\mathcal{\tilde{E}}_F}(u_{r(F,E)})^{\signe{\anyEH}(F)}$$

Let us also define a map $s:\mathcal{L}\rightarrow\mathcal{S}$ which associates arbitrarily, to each long edge, a parallel short edge, and
$$U_{\anyEH}^s=\prod_{F\in\tilde{\settaut{\anyEH}}}\prod_{E\in\mathcal{\tilde{E}}_F}(u_{s\circ\projh(E)})^{\signe{\anyEH}(F)}$$
and remark that $U_{\anyEH}^s=u_{s(A)}$; it is a tautological equality in the free abelian group generated by the symbols $u_i$ associated to short edges. Then, for each $\anyEH\in\trigdimH{1} / \recolh$, one can write the equality
$$u_{s(\anyEH)}=\dfrac{U_{\anyEH}^s}{U_A}=\prod_{F\in\tilde{\settaut{\anyEH}}}\prod_{E\in\mathcal{\tilde{E}}_F}\left(\frac{u_{s\circ\projh(E)}}{u_{r(F,E)}}\right)^{\signe{\anyEH}(F)}.$$
By definitions of $r$ and $s$ and thanks to lemma \ref{Lemme}, there exist words $m_{F,E}\in\Ri$ such that $\frac{u_{s\circ\projh(E)}}{u_{r(F,E)}}=f(m_{F,E})$ and thus there exists a word $m_{s(\anyEH)}$ in \Ri\ such that $f(m_{s(\anyEH)})=u_{s(\anyEH)}$.

According to corollary \ref{corollary}, for any short edge $e_i$ parallel to $s(A)$ and with the generators of \Rh\ $(u_i,v_i)$, there exists a word $m_i\in\Rh$ such that $f(m_i)=u_i$. Besides, for any short edge $e_j$, there exists a short edge $e_i$ such that $v_j=u_i$ or $v_j=u_i^{-1}$.
Thus $f$ is surjective, and this almost allows one to define a ring homomorphism $g:\Rh\rightarrow \Ri$: keeping the previous notation, $g(u_i)=g(v_j)=m_i$. 

One still has \ref{fin enum} things to check:

\begin{enumerate}
\item The images of the $u_i$ do not depend on the choice of $\signe{\anyEH}$:

Let $\signeb{\anyEH}{(1)}$ and $\signeb{\anyEH}{(2)}$ be two maps such that 
$$\anyEH=\sum_F\signeb{\anyEH}{(1)}(F)F=\sum_F\signeb{\anyEH}{(2)}(F)F\in \nchain{1}(M).$$
As $H_2(M,\mathbb{Z})=0$, $\phi=\sum_F\signeb{\anyEH}{(1)}(F)F-\sum_F\signeb{\anyEH}{(2)}(F)F=\bound{3}(\tau)$, where $\tau$ is a linear combination of tetrahedra in \tauh.
If $\tau$ contains the distinguished tetrahedron, then there is an ambiguity in the word $m_{s(\anyEH)}$ given by a power of the product of all shape parameters associated to the edge opposite to the distinguished edge in the distinguished tetraedron, which is $1$ thanks to the gluing equations.
Let us now assume that $\tau$ contains only tetrahedra different from the distinguished one.
Let us denote by $m_{s(\anyEH)}^{(1)}$ and $m_{s(\anyEH)}^{(2)}$ the corresponding words for $\signe{\anyEH}{(1)}$ and $\signe{\anyEH}{(2)}$. Then $\frac{m_{s(\anyEH)}^{(1)}}{m_{s(\anyEH)}^{(2)}}$ is the product of all shape parameters of all tetrahedra contained in $\tau$, so it is equal to $1$.

\item The definition of $U_{\anyEH}$ does not depend on the choice of $\tilde{\settaut{\anyEH}}$ \emph{i.e.} the section $\sigma:\hex\rightarrow\tilde{\hex}$:

Let $F\in\hex$ and $\{F^1, F^2\}=\projh^{-1}(F)$. If one of the $F^i$'s is on the boundary of the distinguished tetrahedron, then there is no problem thanks to the definition of $\sigma$. Else, let $E^i_k, 1\leqslant k \leqslant 3, 1\leqslant i \leqslant 2$ be the long edges in the boundary of $F^i$, with the corresponding shape parameters $z^i_k$. Let $e^i_k=r(F^i, E^i_k)$.
It is then easy to see that for each $k$, $f(z^1_kz^2_k)=u^1_k(u^2_k)^{-1}$, then the ambiguity of $m_{s(A)}$ is given by a power of $\prod_{k=1}^3z^1_kz^2_k=1\in\Ri$.

\item The definition of $m_i$ does not depend on the choice of $s$:

Let $A\in\mathcal{L}$, $e_1, e_2\in\mathcal{S}$ with parameters $(u_1,v_1)$, $(u_2, v_2)$, and $s_1, s_2:\mathcal{L}\rightarrow\mathcal{S}$ such that $s_1(A)=e_1$, $s_2(A)=e_2$ and for any $E\in\mathcal{L}, E\neq A$, $s_1(E)=s_2(E)$. Then, from the definition of $g$, one can write $g_{s_1}(u_1)=m_{s_1(A)}$, $g_{s_2}(u_2)=m_{s_2(A)}$ and $g_{s_2}(u_1)=m_{e_1\rightarrow e_2}m_{s_2(A)}$.

The sum of the powers of $u_i$ in the definition of $U_A^{s_i}$ is 1, so $U_A^{s_1}=\frac{u_1}{u_2}U_A^{s_2}$ and
\begin{align*}
	g_{s_1}(u_1)&=m_{s_1(A)}\\&=m_{e_1\rightarrow e_2}m_{s_2(A)}\\&=g_{s_2}(u_1).
\end{align*}
The same reasonning works for $g_{s_1}(u_2)=g_{s_2}(u_2)$.

For $E\neq A$, the sum of the powers of $u_i$ in the definition of $U_E^{s_i}$ is 0, so changing $s$ has no influence on the definition of $g(u)$ for the short edges parallel with E.

\item Let us now prove that $g$ respects the relations in \Rh.

The hexagonal face relations are preserved by definition of the image of $v_i$.

Until now, orientation of long edges, hence the sign, was implied in the notation, and short edges were considered parallel with long edges. Reversing orientation of a short edge naturally inverses the image of its generator $u$ by $g$.

Let us check now, that reversing orientation preserves the relation for the parameter $v$.

Let $e_1, e_2, e_3$ be three short edges on the boundary of a face $F\in\mathcal{\tilde{H}}$, respectively parallel to long edges
 $E_1, E_2, E_3$ (also on the boundary of $F$) and oriented such that the starting point of $e_1$ and $e_3$ bound $E_2$, and the end point of $e_1$ and $e_2$ bound $E_3$.
Then, by definition of $g$, $g(v_2)=g(u_1)$. Let us reverse orientation of $e_2$ and let us call it $\bar{e_2}$, then $g(v_{\bar{2}})=g(u_3)$, and one has to check that $g(u_3)=-g(u_2)^{-1}g(u_1)$. Let us do this from the definition of $g(u_3)$ and let us assume without loss of generality, that $e_3=s(E_3)$. Let us write $\bound{2}(F)=E_2+E_3-E_1\in\nchain{1}(M)$, so 
\begin{align*}
	E_3&=\bound{2}(F)+E_1-E_2\\&=\bound{2}\left(F+\sum_{\phi}\signe{E_1}(\phi) \phi-\sum_{\phi}\signe{E_2}(\phi) \phi\right).
\end{align*}

Let us assume, as a first case, that $\projh(E_1), \projh(E_2), \projh(E_3)$ are three different edges and that $e_1=s(E_1)$ and $ e_2=s(E_2)$. Let us denote $u_i'=r(F,E_i)$. Then, one has the following tautological equality:
$$u_3=\frac{u_3 u_1' u_2}{u_3' u_1 u_2'}\frac{U_1^s}{U_1}\frac{U_2}{U_2^s}.$$ 
$\frac{u_3 u_1' u_2}{u_3' u_1 u_2'}$ has as preimage in $\Ri$ a product of three different shape parameters of a tetrahedron (which is equal to -1), $\frac{U_1^s}
{U_1}$ has as preimage in \Ri\ a word $m_1$ from lemma \ref{Lemme} such that $f(m_1)=u_1$ and $\frac{U_2^s}{U_2}$ has as preimage in \Ri\ a word $m_2$ from lemma \ref{Lemme} such that $f(m_2)=u_2$. So, by construction of $g$, $g(u_3)=-g(u_1)g(u_2)^{-1}$.
Let us now assume  that another edge of the boundary of $F$ has projection equal to $E_3$ in $\trigdimH{1} / \recolh$, for example $E_1$, with $s(E_1)=s(E_3)=e_3$. Then, one has the following tautological equality: 
$u_3=\frac{u_3 u_1' u_2}{u_3' u_3 u_2'}\frac{U_1^s}{U_1}\frac{U_2}{U_2^s}$, but now, $\frac{U_1^s}{U_1}=u_3$ and
$\frac{u_3 u_1' u_2}{u_3' u_3 u_2'}=\frac{u_3 u_1' u_2}{u_3' u_1 u_2'}\frac{u_1}{u_3}$ which has preimage in \Ri\, the product of three different shape parameters of a tetrahedron and the word of corollary \ref{corollary} going from $e_3$ to $e_1$.

Let us end with the triangular face relations. Let $e_i, e_j, e_k$ be short edges bounding a triangular face in $\tauhtron$, with generators $(u_i,v_i),(u_j,v_j),(u_k,v_k)$. Let $F$ be the opposite hexagonal face, with long edges $E_i,E_j,E_k$ on its boundary.
Let us assume they are oriented so that $\bound{2}(F)=E_i+E_j-E_k$, and one has to prove that $g(u_i)g(u_j)g(u_k)^{ -1}=1$: the reasoning is similar to the previous one. One also has to prove that 
$$-g(v_k)+g(u_i)g(v_j)+g(v_i)=0,$$ 
or equivalently, 
$$g(u_i)g(v_j)g(v_k)^{-1}+g(v_i)g(v_k)^{-1}=1,$$
so, by definition of $g$, 
$$g(u_i)g(u_{\check{j}})g(u_{\check{k}})^{-1}+g(u_{\check{i}})g(u_{\check{k}})^{-1}=1.$$
Let us call $e_l$, the short edge in the same tetrahedron parallel with $e_i$, then it is equivalent to prove that $$g(u_i)g(u_l)^{-1}g(u_l)g(u_{\check{j}})g(u_{\check{k}})^{-1}+g(u_{\check{i}})g(u_{\check{k}})^{-1}=1.$$
The left part of the equality is equal, by definition of $g$ and thanks to the preceding relation, to $-z_kz_j+z_j=1$.
\label{fin enum}

\end{enumerate}

It is now a straightforward computation to check that $f\circ g =id_{\Ri}$ and $g\circ f=id_{\Rh}$.

\begin{remark}
	We actually proved that for any edge $E$ of $\nchain{1}(M)$ of which the projection onto $H_1(M)$ is trivial, the parameters $u_i$ associated to short edges parallel to $E$ are in the image of $f$.
\end{remark}

\section{Examples} \label{examples}

\subsection{Figure-eight knot}

Let $K$ be the figure-eight knot embedded in $\mathbb{S}^3$. The cell complex of figure \ref{figure eight knot} is a truncated particular H-triangulation corresponding to the well known ideal triangulation of $\mathbb{S}^3\setminus K$, see [3].
The distinguished edge is dotted. The edges labelled with double and triple arrows are identified when collapsing that edge onto the vertex before truncation. The shape parameters $z_i$ and $w_i$ are actually associated to the corresponding edges of the ideal triangulation. 

Thurston's gluing equations are $z_1^2z_2w_1w_3^2=1$ and $z_2z_3^2w_1w_2^2=1$, or, equivalently, $z_1(1-z_1)w_3(1-w_3)=1$.

Keeping the previous notations, one has, for example, $f(z_1)=v_2v_{11}^{-1}=u_{7}u_{16}^{-1}$. According to theorem \ref{theorem}, $f$ is an isomorphism. Let us see, for example, what the preimage of $u_1$ is. $e_1$ is parallel to the edge labelled with a simple arrow. That edge is tautologically equal to $\bound{2}(B)\in\nchain{1}(\mathbb{S}^3)$.
Let us consider the only triangular face opposite to the hexagonal face corresponding to $B$ which is not in the distinguished tetrahedron. Let us assume that the image by $s$ of the long edge labelled with a double arrow is $e_{15}$, and $e_1$ for the long one labelled with a simple arrow. Then, one has the following equality : $u_1=\frac{u_1 u_{15} u_{15}^{-1}}{u_{14} u_{15} u_{16}^{-1}}$.
Besides, following the proof of lemma \ref{Lemme}, one has $f(w_2z_2w_2z_3w_1)=\frac{u_1}{u_{14}}$ and, by definition of $f$ and thanks to the hexagonal face relations in \Rh,  $f(w_3)=\frac{u_{16}}{u_{15}}$. Then, using the gluing equations in \Ri, we find that $f(z_2z_3w_1w_{2}^{2}w_3)=u_1$.

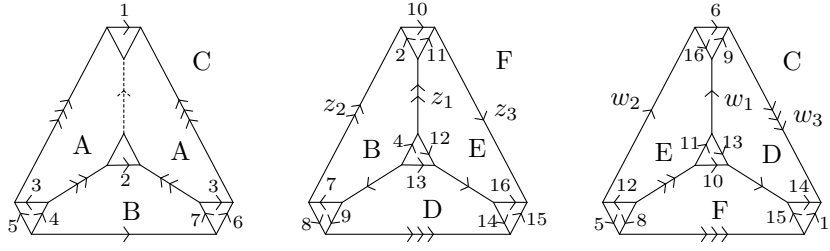
\begin{figure}
	\begin{center}
\begin{tikzpicture}[line cap=round,line join=round,>=triangle 45,x=0.2761812170441546cm,y=0.3046495887829334cm]
\clip(-0.8,-7.48) rectangle (39.02,4.33);
\draw (0.6,-6)-- (9.4,-6);
\draw (5.25,-6) -- (5,-6.32);
\draw (5.25,-6) -- (5,-5.68);
\draw (4.2,-3)-- (5.77,-2.98);
\draw (5.24,-2.99) -- (4.99,-3.31);
\draw (5.24,-2.99) -- (4.99,-2.67);
\draw (4.2,3.01)-- (5.8,3.01);
\draw (5.25,3.01) -- (5,2.69);
\draw (5.25,3.01) -- (5,3.32);
\draw (18.23,-2.98)-- (15.4,-4.61);
\draw (16.6,-3.92) -- (16.65,-3.52);
\draw (16.6,-3.92) -- (16.97,-4.07);
\draw (19.8,-3)-- (22.6,-4.61);
\draw (21.41,-3.93) -- (21.04,-4.08);
\draw (21.41,-3.93) -- (21.36,-3.53);
\draw (22.6,-4.61)-- (23.4,-6);
\draw (23.12,-5.52) -- (22.72,-5.47);
\draw (23.12,-5.52) -- (23.28,-5.15);
\draw (19,1.62)-- (18.2,3.01);
\draw (18.48,2.53) -- (18.88,2.47);
\draw (18.48,2.53) -- (18.32,2.15);
\draw (15.4,-4.61)-- (14.6,-6);
\draw (14.88,-5.52) -- (14.72,-5.15);
\draw (14.88,-5.52) -- (15.28,-5.47);
\draw (19,1.62)-- (19.8,3.01);
\draw (19.52,2.53) -- (19.68,2.15);
\draw (19.52,2.53) -- (19.12,2.47);
\draw (33,-1.62)-- (33,1.62);
\draw (33,0.25) -- (33.32,0);
\draw (33,0.25) -- (32.68,0);
\draw (29.4,-4.61)-- (27.8,-4.61);
\draw (28.35,-4.61) -- (28.6,-4.3);
\draw (28.35,-4.61) -- (28.6,-4.93);
\draw (36.6,-4.61)-- (38.2,-4.61);
\draw (37.65,-4.61) -- (37.4,-4.93);
\draw (37.65,-4.61) -- (37.4,-4.3);
\draw (27.8,-4.61)-- (32.2,3.01);
\draw (30.12,-0.59) -- (30.28,-0.96);
\draw (30.12,-0.59) -- (29.72,-0.64);
\draw (32.2,-3)-- (33,-1.62);
\draw (32.73,-2.09) -- (32.88,-2.46);
\draw (32.73,-2.09) -- (32.33,-2.15);
\draw (37.4,-6)-- (38.2,-4.61);
\draw (37.92,-5.09) -- (38.08,-5.47);
\draw (37.92,-5.09) -- (37.52,-5.15);
\draw (33.77,-2.98)-- (36.6,-4.61);
\draw (35.4,-3.92) -- (35.03,-4.07);
\draw (35.4,-3.92) -- (35.35,-3.52);
\draw (29.4,-4.61)-- (28.6,-6);
\draw (28.88,-5.52) -- (28.72,-5.15);
\draw (28.88,-5.52) -- (29.28,-5.47);
\draw (33,1.62)-- (33.8,3.01);
\draw (33.52,2.53) -- (33.68,2.15);
\draw (33.52,2.53) -- (33.12,2.47);
\draw (1.4,-4.61)-- (4.2,-3);
\draw (3.23,-3.56) -- (3.17,-3.96);
\draw (3.23,-3.56) -- (2.86,-3.41);
\draw (2.8,-3.81) -- (2.75,-4.2);
\draw (2.8,-3.81) -- (2.43,-3.65);
\draw (8.6,-4.61)-- (5.77,-2.98);
\draw (6.76,-3.55) -- (7.13,-3.4);
\draw (6.76,-3.55) -- (6.81,-3.95);
\draw (7.19,-3.8) -- (7.56,-3.65);
\draw (7.19,-3.8) -- (7.24,-4.2);
\draw (13.8,-4.61)-- (18.2,3.01);
\draw (16.25,-0.38) -- (16.4,-0.75);
\draw (16.25,-0.38) -- (15.85,-0.43);
\draw (16,-0.8) -- (16.15,-1.18);
\draw (16,-0.8) -- (15.6,-0.86);
\draw (19,-1.62)-- (19,1.62);
\draw (19,0.5) -- (19.32,0.25);
\draw (19,0.5) -- (18.68,0.25);
\draw (19,0) -- (19.32,-0.24);
\draw (19,0) -- (18.68,-0.24);
\draw (29.4,-4.61)-- (32.2,-3);
\draw (31.23,-3.56) -- (31.17,-3.96);
\draw (31.23,-3.56) -- (30.86,-3.41);
\draw (30.8,-3.81) -- (30.75,-4.2);
\draw (30.8,-3.81) -- (30.43,-3.65);
\draw (9.4,-6)-- (8.6,-4.61);
\draw (8.88,-5.09) -- (9.28,-5.15);
\draw (8.88,-5.09) -- (8.72,-5.47);
\draw (4.2,3.01)-- (5,1.62);
\draw (0.6,-6)-- (1.4,-4.61);
\draw (1.12,-5.09) -- (1.28,-5.47);
\draw (1.12,-5.09) -- (0.72,-5.15);
\draw (5.8,3.01)-- (5,1.62);
\draw (15.4,-4.61)-- (13.8,-4.61);
\draw (14.35,-4.61) -- (14.6,-4.3);
\draw (14.35,-4.61) -- (14.6,-4.93);
\draw (22.6,-4.61)-- (24.2,-4.61);
\draw (23.65,-4.61) -- (23.4,-4.93);
\draw (23.65,-4.61) -- (23.4,-4.3);
\draw (18.23,-2.98)-- (19,-1.62);
\draw (18.73,-2.08) -- (18.89,-2.46);
\draw (18.73,-2.08) -- (18.34,-2.14);
\draw (23.4,-6)-- (24.2,-4.61);
\draw (23.92,-5.09) -- (24.08,-5.47);
\draw (23.92,-5.09) -- (23.52,-5.15);
\draw (37.4,-6)-- (36.6,-4.61);
\draw (36.88,-5.09) -- (37.28,-5.15);
\draw (36.88,-5.09) -- (36.72,-5.47);
\draw (32.2,3.01)-- (33,1.62);
\draw (32.72,2.1) -- (32.32,2.15);
\draw (32.72,2.1) -- (32.88,2.47);
\draw (-0.2,-4.61)-- (4.2,3.01);
\draw (2.12,-0.59) -- (2.28,-0.96);
\draw (2.12,-0.59) -- (1.72,-0.64);
\draw (1.88,-1.02) -- (2.03,-1.39);
\draw (1.88,-1.02) -- (1.48,-1.07);
\draw (2.37,-0.16) -- (2.52,-0.53);
\draw (2.37,-0.16) -- (1.97,-0.22);
\draw (4.2,-3)-- (5,-1.62);
\draw (9.4,-6)-- (10.2,-4.61);
\draw (9.92,-5.09) -- (10.08,-5.47);
\draw (9.92,-5.09) -- (9.52,-5.15);
\draw (10.2,-4.61)-- (5.8,3.01);
\draw (7.88,-0.59) -- (8.28,-0.64);
\draw (7.88,-0.59) -- (7.72,-0.96);
\draw (8.12,-1.02) -- (8.52,-1.07);
\draw (8.12,-1.02) -- (7.97,-1.39);
\draw (7.63,-0.16) -- (8.03,-0.22);
\draw (7.63,-0.16) -- (7.48,-0.53);
\draw (5.77,-2.98)-- (5,-1.62);
\draw (0.6,-6)-- (-0.2,-4.61);
\draw (0.08,-5.09) -- (0.48,-5.15);
\draw (0.08,-5.09) -- (-0.08,-5.47);
\draw (14.6,-6)-- (23.4,-6);
\draw (19.25,-6) -- (19,-6.32);
\draw (19.25,-6) -- (19,-5.68);
\draw (18.75,-6) -- (18.51,-6.32);
\draw (18.75,-6) -- (18.51,-5.68);
\draw (19.74,-6) -- (19.49,-6.32);
\draw (19.74,-6) -- (19.49,-5.68);
\draw (18.23,-2.98)-- (19.8,-3);
\draw (19.26,-2.99) -- (19.01,-3.31);
\draw (19.26,-2.99) -- (19.01,-2.67);
\draw (18.2,3.01)-- (19.8,3.01);
\draw (19.25,3.01) -- (19,2.69);
\draw (19.25,3.01) -- (19,3.32);
\draw (28.6,-6)-- (37.4,-6);
\draw (33.25,-6) -- (33,-6.32);
\draw (33.25,-6) -- (33,-5.68);
\draw (32.75,-6) -- (32.51,-6.32);
\draw (32.75,-6) -- (32.51,-5.68);
\draw (33.74,-6) -- (33.49,-6.32);
\draw (33.74,-6) -- (33.49,-5.68);
\draw (32.2,-3)-- (33.77,-2.98);
\draw (33.24,-2.99) -- (32.99,-3.31);
\draw (33.24,-2.99) -- (32.99,-2.67);
\draw (32.2,3.01)-- (33.8,3.01);
\draw (33.25,3.01) -- (33,2.69);
\draw (33.25,3.01) -- (33,3.32);
\draw (33.8,3.01)-- (38.2,-4.61);
\draw (36.12,-1.02) -- (35.72,-0.96);
\draw (36.12,-1.02) -- (36.28,-0.64);
\draw (35.88,-0.59) -- (35.48,-0.53);
\draw (35.88,-0.59) -- (36.03,-0.22);
\draw (36.37,-1.45) -- (35.97,-1.39);
\draw (36.37,-1.45) -- (36.52,-1.07);
\draw (33,-1.62)-- (33.77,-2.98);
\draw (33.51,-2.51) -- (33.11,-2.46);
\draw (33.51,-2.51) -- (33.66,-2.14);
\draw (27.8,-4.61)-- (28.6,-6);
\draw (28.32,-5.52) -- (27.92,-5.47);
\draw (28.32,-5.52) -- (28.48,-5.15);
\draw (19.8,3.01)-- (24.2,-4.61);
\draw (22.12,-1.02) -- (21.72,-0.96);
\draw (22.12,-1.02) -- (22.28,-0.64);
\draw (19,-1.62)-- (19.8,-3);
\draw (19.52,-2.52) -- (19.12,-2.46);
\draw (19.52,-2.52) -- (19.67,-2.15);
\draw (13.8,-4.61)-- (14.6,-6);
\draw (14.32,-5.52) -- (13.92,-5.47);
\draw (14.32,-5.52) -- (14.48,-5.15);
\draw [dash pattern=on 1pt off 1pt] (5,-1.62)-- (5,1.62);
\draw [dash pattern=on 1pt off 1pt] (5,0.25) -- (5.32,0);
\draw [dash pattern=on 1pt off 1pt] (5,0.25) -- (4.68,0);
\draw (1.4,-4.61)-- (-0.2,-4.61);
\draw (0.35,-4.61) -- (0.6,-4.3);
\draw (0.35,-4.61) -- (0.6,-4.93);
\draw (8.6,-4.61)-- (10.2,-4.61);
\draw (9.65,-4.61) -- (9.4,-4.93);
\draw (9.65,-4.61) -- (9.4,-4.3);
\draw (19.16,0.57) node[anchor=north west] {$z_1$};
\draw (22.17,0.05) node[anchor=north west] {$z_3$};
\draw (14.03,0.28) node[anchor=north west] {$z_2$};
\draw (33.14,0.47) node[anchor=north west] {$w_1$};
\draw (36.43,-0.23) node[anchor=north west] {$w_3$};
\draw (27.73,0.52) node[anchor=north west] {$w_2$};
\draw (2.11,-1.32) node[anchor=north west] {A};
\draw (6.78,-1.55) node[anchor=north west] {A};
\draw (4.47,-4.19) node[anchor=north west] {B};
\draw (7.81,2.36) node[anchor=north west] {C};
\draw (15.91,-1.5) node[anchor=north west] {B};
\draw (35.92,2.36) node[anchor=north west] {C};
\draw (18.78,-4.09) node[anchor=north west] {D};
\draw (34.93,-1.69) node[anchor=north west] {D};
\draw (20.95,-1.46) node[anchor=north west] {E};
\draw (29.84,-1.6) node[anchor=north west] {E};
\draw (22.26,2.4) node[anchor=north west] {F};
\draw (32.53,-4.14) node[anchor=north west] {F};
\begin{scriptsize}
\draw[color=black] (5.08,-3.62) node {2};
\draw[color=black] (5.03,3.77) node {1};
\draw[color=black] (22.31,-5.41) node {14};
\draw[color=black] (18.17,1.79) node {2};
\draw[color=black] (15.58,-5.22) node {9};
\draw[color=black] (19.91,1.79) node {11};
\draw[color=black] (28.95,-3.86) node {12};
\draw[color=black] (37.14,-3.91) node {14};
\draw[color=black] (31.96,-2.02) node {11};
\draw[color=black] (38.37,-5.51) node {1};
\draw[color=black] (29.66,-5.36) node {8};
\draw[color=black] (33.8,1.7) node {9};
\draw[color=black] (8.42,-5.36) node {7};
\draw[color=black] (1.69,-5.32) node {4};
\draw[color=black] (14.87,-3.81) node {7};
\draw[color=black] (22.92,-3.86) node {16};
\draw[color=black] (18.07,-1.93) node {4};
\draw[color=black] (24.67,-5.32) node {15};
\draw[color=black] (36.06,-5.27) node {15};
\draw[color=black] (32.15,1.79) node {16};
\draw[color=black] (10.45,-5.46) node {6};
\draw[color=black] (-0.29,-5.65) node {5};
\draw[color=black] (18.92,-3.72) node {13};
\draw[color=black] (18.97,3.82) node {10};
\draw[color=black] (33.05,-3.67) node {10};
\draw[color=black] (33.19,3.77) node {6};
\draw[color=black] (33.99,-1.98) node {13};
\draw[color=black] (27.63,-5.55) node {5};
\draw[color=black] (20.05,-1.83) node {12};
\draw[color=black] (13.7,-5.55) node {8};
\draw[color=black] (0.84,-3.86) node {3};
\draw[color=black] (9.27,-3.91) node {3};
\end{scriptsize}
\end{tikzpicture}
\caption{Figure eight knot in $\mathbb{S}^3$ \label{figure eight knot}}
\end{center}
\end{figure}

\subsection{A knot in $\mathbb{S}^2\times\mathbb{S}^1$}

Figure \ref{essedeuxcroixessun} represents a truncated particular H-triangulation of a knot in $\mathbb{S}^2\times\mathbb{S}^1$. The knot is represented by the dotted edge. Here, $M=\mathbb{S}^2\times\mathbb{S}^1$, so $H_1(M,\mathbb{Z})\simeq H_2(M,\mathbb{Z})\simeq \mathbb{Z}$. 

The shape parameters $z$ and $w$ are associated to the corresponding edges in \taui. The Thurston's  gluing equation is $zw=1$. After computation, the presentation of \Rh\ reduces to $\mathbb{Z}\langle u_1,u_{\bar{1}},u_2 \rvert u_1u_{\bar{1}}=1,u_2^3=1\rangle$

$f$ is neither surjective, nor injective. Indeed, $u_1\notin\text{Im}(f)$, and $f(w)=u_2$, so $f(w^3)=1$.

\begin{figure}
	\begin{center}
\begin{tikzpicture}[line cap=round,line join=round,>=triangle 45,x=0.595370899485691cm,y=0.7110996878507736cm]
\clip(-6,2.46) rectangle (12.47,7.38);
\draw (-2.75,6.9)-- (-3.75,6.9);
\draw (-3.38,6.9) -- (-3.25,7.07);
\draw (-3.38,6.9) -- (-3.25,6.74);
\draw (-3.25,6.04)-- (-3.75,6.9);
\draw (-3.56,6.58) -- (-3.35,6.55);
\draw (-3.56,6.58) -- (-3.64,6.39);
\draw (-3.25,6.04)-- (-2.75,6.9);
\draw (-2.93,6.58) -- (-2.85,6.39);
\draw (-2.93,6.58) -- (-3.14,6.55);
\draw (-5.62,3.65)-- (-5.12,2.78);
\draw (-5.31,3.1) -- (-5.51,3.13);
\draw (-5.31,3.1) -- (-5.23,3.3);
\draw (-5.62,3.65)-- (-4.62,3.65);
\draw (-4.99,3.65) -- (-5.12,3.48);
\draw (-4.99,3.65) -- (-5.12,3.81);
\draw (-4.62,3.65)-- (-5.12,2.78);
\draw (-4.93,3.1) -- (-5.01,3.3);
\draw (-4.93,3.1) -- (-4.73,3.13);
\draw (-0.86,3.65)-- (-1.36,2.79);
\draw (-1.18,3.11) -- (-1.26,3.3);
\draw (-1.18,3.11) -- (-0.97,3.14);
\draw (-1.86,3.65)-- (-1.36,2.79);
\draw (-1.55,3.11) -- (-1.76,3.14);
\draw (-1.55,3.11) -- (-1.47,3.3);
\draw (-3.24,5.03)-- (-3.75,4.15);
\draw (-3.56,4.48) -- (-3.64,4.67);
\draw (-3.56,4.48) -- (-3.35,4.51);
\draw (-2.74,4.16)-- (-3.75,4.15);
\draw (-3.37,4.16) -- (-3.24,4.32);
\draw (-3.37,4.16) -- (-3.24,3.99);
\draw (-3.24,5.03)-- (-2.74,4.16);
\draw (-2.93,4.48) -- (-3.14,4.51);
\draw (-2.93,4.48) -- (-2.85,4.67);
\draw (-3.75,6.9)-- (-5.62,3.65);
\draw (-4.75,5.16) -- (-4.83,5.36);
\draw (-4.75,5.16) -- (-4.54,5.19);
\draw (-2.75,6.9)-- (-0.86,3.65);
\draw (-1.74,5.17) -- (-1.95,5.19);
\draw (-1.74,5.17) -- (-1.66,5.36);
\draw [dash pattern=on 1pt off 1pt] (-3.25,6.04)-- (-3.24,5.03);
\draw [dash pattern=on 1pt off 1pt] (-3.24,5.4) -- (-3.41,5.53);
\draw [dash pattern=on 1pt off 1pt] (-3.24,5.4) -- (-3.08,5.53);
\draw (-1.36,2.79)-- (-5.12,2.78);
\draw (-3.37,2.78) -- (-3.24,2.95);
\draw (-3.37,2.78) -- (-3.24,2.62);
\draw (-3.11,2.78) -- (-2.98,2.95);
\draw (-3.11,2.78) -- (-2.98,2.62);
\draw (-3.63,2.78) -- (-3.5,2.95);
\draw (-3.63,2.78) -- (-3.5,2.62);
\draw (3.75,6.95)-- (2.75,6.96);
\draw (3.12,6.95) -- (3.25,7.12);
\draw (3.12,6.95) -- (3.25,6.79);
\draw (1.32,2.78)-- (0.82,3.64);
\draw (1.01,3.32) -- (1.22,3.29);
\draw (1.01,3.32) -- (0.93,3.13);
\draw (5.15,2.76)-- (5.65,3.63);
\draw (5.47,3.31) -- (5.55,3.11);
\draw (5.47,3.31) -- (5.26,3.28);
\draw (7.83,2.79)-- (7.33,3.66);
\draw (7.51,3.34) -- (7.72,3.31);
\draw (7.51,3.34) -- (7.43,3.14);
\draw (11.66,2.8)-- (12.16,3.67);
\draw (11.97,3.35) -- (12.05,3.15);
\draw (11.97,3.35) -- (11.76,3.32);
\draw (10.23,6.98)-- (9.23,6.98);
\draw (9.6,6.98) -- (9.73,7.15);
\draw (9.6,6.98) -- (9.73,6.81);
\draw (1.82,3.64)-- (1.32,2.78);
\draw (1.51,3.1) -- (1.43,3.29);
\draw (1.51,3.1) -- (1.72,3.13);
\draw (1.82,3.64)-- (0.82,3.64);
\draw (1.2,3.64) -- (1.33,3.81);
\draw (1.2,3.64) -- (1.32,3.48);
\draw (4.65,3.63)-- (5.15,2.76);
\draw (4.97,3.09) -- (4.76,3.12);
\draw (4.97,3.09) -- (5.05,3.28);
\draw (4.65,3.63)-- (5.65,3.63);
\draw (5.28,3.63) -- (5.15,3.46);
\draw (5.28,3.63) -- (5.16,3.8);
\draw (3.74,4.16)-- (3.24,5.03);
\draw (3.43,4.71) -- (3.64,4.68);
\draw (3.43,4.71) -- (3.35,4.51);
\draw (2.74,4.17)-- (3.24,5.03);
\draw (3.06,4.71) -- (3.14,4.52);
\draw (3.06,4.71) -- (2.85,4.68);
\draw (3.74,4.16)-- (2.74,4.17);
\draw (3.11,4.16) -- (3.24,4.33);
\draw (3.11,4.16) -- (3.24,4);
\draw (3.25,6.09)-- (2.75,6.96);
\draw (2.93,6.63) -- (3.14,6.6);
\draw (2.93,6.63) -- (2.85,6.44);
\draw (3.25,6.09)-- (3.75,6.95);
\draw (3.56,6.63) -- (3.64,6.44);
\draw (3.56,6.63) -- (3.35,6.6);
\draw (9.24,4.19)-- (9.74,5.06);
\draw (9.55,4.74) -- (9.63,4.54);
\draw (9.55,4.74) -- (9.34,4.71);
\draw (10.24,4.19)-- (9.74,5.06);
\draw (9.92,4.74) -- (10.13,4.71);
\draw (9.92,4.74) -- (9.84,4.54);
\draw (10.24,4.19)-- (9.24,4.19);
\draw (9.61,4.19) -- (9.74,4.36);
\draw (9.61,4.19) -- (9.74,4.02);
\draw (7.33,3.66)-- (8.33,3.66);
\draw (7.96,3.66) -- (7.83,3.49);
\draw (7.96,3.66) -- (7.83,3.82);
\draw (7.83,2.79)-- (8.33,3.66);
\draw (8.14,3.34) -- (8.22,3.14);
\draw (8.14,3.34) -- (7.93,3.31);
\draw (11.66,2.8)-- (11.16,3.67);
\draw (11.34,3.35) -- (11.55,3.32);
\draw (11.34,3.35) -- (11.26,3.15);
\draw (11.16,3.67)-- (12.16,3.67);
\draw (11.79,3.67) -- (11.66,3.5);
\draw (11.79,3.67) -- (11.66,3.83);
\draw (9.23,6.98)-- (9.73,6.11);
\draw (9.55,6.43) -- (9.34,6.46);
\draw (9.55,6.43) -- (9.63,6.63);
\draw (10.23,6.98)-- (9.73,6.11);
\draw (9.92,6.43) -- (9.84,6.63);
\draw (9.92,6.43) -- (10.13,6.46);
\draw (0.82,3.64)-- (2.75,6.96);
\draw (1.85,5.41) -- (1.93,5.22);
\draw (1.85,5.41) -- (1.64,5.38);
\draw (5.65,3.63)-- (3.75,6.95);
\draw (4.64,5.4) -- (4.85,5.37);
\draw (4.64,5.4) -- (4.56,5.21);
\draw (2.74,4.17)-- (1.82,3.64);
\draw (2.17,3.84) -- (2.2,4.05);
\draw (2.17,3.84) -- (2.37,3.76);
\draw (3.74,4.16)-- (4.65,3.63);
\draw (4.31,3.83) -- (4.12,3.75);
\draw (4.31,3.83) -- (4.28,4.04);
\draw (3.24,5.03)-- (3.25,6.09);
\draw (3.25,5.82) -- (3.41,5.69);
\draw (3.25,5.82) -- (3.08,5.69);
\draw (3.25,5.56) -- (3.41,5.43);
\draw (3.25,5.56) -- (3.08,5.43);
\draw (5.15,2.76)-- (1.32,2.78);
\draw (3.11,2.77) -- (3.24,2.94);
\draw (3.11,2.77) -- (3.24,2.6);
\draw (3.37,2.77) -- (3.5,2.94);
\draw (3.37,2.77) -- (3.5,2.6);
\draw (2.85,2.77) -- (2.98,2.94);
\draw (2.85,2.77) -- (2.98,2.6);
\draw (11.66,2.8)-- (7.83,2.79);
\draw (9.61,2.8) -- (9.74,2.96);
\draw (9.61,2.8) -- (9.74,2.63);
\draw (9.87,2.8) -- (10,2.96);
\draw (9.87,2.8) -- (10,2.63);
\draw (9.35,2.79) -- (9.48,2.96);
\draw (9.35,2.79) -- (9.48,2.63);
\draw (7.33,3.66)-- (9.23,6.98);
\draw (8.41,5.54) -- (8.49,5.35);
\draw (8.41,5.54) -- (8.2,5.51);
\draw (8.28,5.32) -- (8.36,5.12);
\draw (8.28,5.32) -- (8.07,5.29);
\draw (12.16,3.67)-- (10.23,6.98);
\draw (11.06,5.55) -- (11.27,5.52);
\draw (11.06,5.55) -- (10.98,5.35);
\draw (11.19,5.32) -- (11.4,5.3);
\draw (11.19,5.32) -- (11.11,5.13);
\draw (9.74,5.06)-- (9.73,6.11);
\draw (9.73,5.72) -- (9.9,5.59);
\draw (9.73,5.72) -- (9.57,5.59);
\draw (8.33,3.66)-- (9.24,4.19);
\draw (8.89,3.99) -- (8.87,3.78);
\draw (8.89,3.99) -- (8.7,4.07);
\draw (11.16,3.67)-- (10.24,4.19);
\draw (10.59,3.99) -- (10.78,4.07);
\draw (10.59,3.99) -- (10.62,3.78);
\draw (-0.86,3.65)-- (-1.86,3.65);
\draw (-1.49,3.65) -- (-1.36,3.82);
\draw (-1.49,3.65) -- (-1.36,3.49);
\draw (-4.62,3.65)-- (-3.75,4.15);
\draw (-3.96,4.03) -- (-3.99,3.82);
\draw (-3.96,4.03) -- (-4.15,4.11);
\draw (-4.18,3.9) -- (-4.21,3.69);
\draw (-4.18,3.9) -- (-4.38,3.98);
\draw (-1.86,3.65)-- (-2.74,4.16);
\draw (-2.53,4.03) -- (-2.33,4.11);
\draw (-2.53,4.03) -- (-2.5,3.82);
\draw (-2.3,3.9) -- (-2.11,3.98);
\draw (-2.3,3.9) -- (-2.27,3.69);
\begin{scriptsize}
\draw[color=black] (-4.44,4.76) node {A};
\draw[color=black] (-2.09,4.71) node {A};
\draw[color=black] (-3.28,3.42) node {E};
\draw[color=black] (-1.87,6.14) node {B};
\draw[color=black] (2.09,4.73) node {C};
\draw[color=black] (4.56,4.63) node {D};
\draw[color=black] (3.15,3.45) node {B};
\draw[color=black] (8.49,4.73) node {C};
\draw[color=black] (10.92,4.73) node {D};
\draw[color=black] (4.68,6.04) node {F};
\draw[color=black] (9.65,3.42) node {F};
\draw[color=black] (11.21,6.07) node {E};
\draw[color=black] (3.3,7.18) node {$e_2$};
\draw[color=black] (2.73,6.39) node {$e_1$};
\draw[color=black] (9.9,3.87) node {$e_2$};
\draw[color=black] (9.23,6.41) node {$e_1$};
\draw[color=black] (3.57,5.33) node {$z$};
\draw[color=black] (10,5.38) node {$w$};
\end{scriptsize}
\end{tikzpicture}
	\end{center}
\caption{Knot in $\mathbb{S}^2\times\mathbb{S}^1$. All edges are oriented, but for simplicity, only the necessary names have been given. \label{essedeuxcroixessun}}
\end{figure}
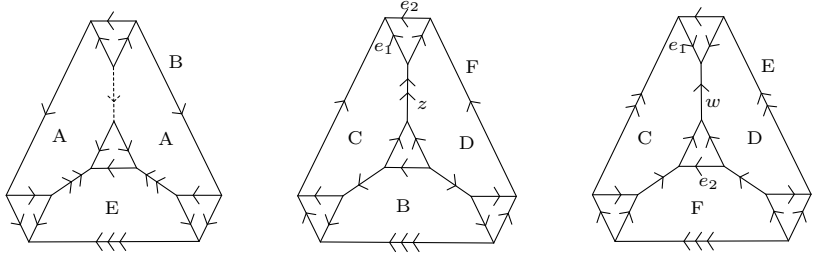

\subsection{Knot in $L(3,1)$}

\begin{figure}
\begin{center}
\begin{tikzpicture}[line cap=round,line join=round,>=triangle 45,x=0.566616726077744cm,y=0.5667578510406542cm]
\clip(-6.36,-3.13) rectangle (13.05,8.16);
\draw (-2.75,6.9)-- (-3.75,6.9);
\draw (-3.38,6.9) -- (-3.25,7.08);
\draw (-3.38,6.9) -- (-3.25,6.73);
\draw (-3.25,6.04)-- (-3.75,6.9);
\draw (-3.56,6.59) -- (-3.34,6.56);
\draw (-3.56,6.59) -- (-3.65,6.38);
\draw (-3.25,6.04)-- (-2.75,6.9);
\draw (-2.93,6.59) -- (-2.84,6.38);
\draw (-2.93,6.59) -- (-3.15,6.56);
\draw (-5.62,3.65)-- (-5.12,2.78);
\draw (-5.3,3.1) -- (-5.52,3.13);
\draw (-5.3,3.1) -- (-5.22,3.3);
\draw (-5.62,3.65)-- (-4.62,3.65);
\draw (-4.99,3.65) -- (-5.12,3.47);
\draw (-4.99,3.65) -- (-5.12,3.82);
\draw (-4.62,3.65)-- (-5.12,2.78);
\draw (-4.94,3.1) -- (-5.02,3.3);
\draw (-4.94,3.1) -- (-4.72,3.13);
\draw (-0.86,3.65)-- (-1.36,2.79);
\draw (-1.18,3.1) -- (-1.26,3.31);
\draw (-1.18,3.1) -- (-0.96,3.13);
\draw (-1.86,3.65)-- (-1.36,2.79);
\draw (-1.55,3.1) -- (-1.76,3.13);
\draw (-1.55,3.1) -- (-1.46,3.31);
\draw (-3.24,5.03)-- (-3.75,4.15);
\draw (-3.56,4.47) -- (-3.65,4.68);
\draw (-3.56,4.47) -- (-3.34,4.5);
\draw (-2.74,4.16)-- (-3.75,4.15);
\draw (-3.38,4.16) -- (-3.24,4.33);
\draw (-3.38,4.16) -- (-3.24,3.98);
\draw (-3.24,5.03)-- (-2.74,4.16);
\draw (-2.92,4.47) -- (-3.14,4.5);
\draw (-2.92,4.47) -- (-2.84,4.68);
\draw [dash pattern=on 2pt off 2pt] (-3.25,6.04)-- (-3.24,5.03);
\draw [dash pattern=on 2pt off 2pt] (-3.24,5.4) -- (-3.42,5.53);
\draw [dash pattern=on 2pt off 2pt] (-3.24,5.4) -- (-3.07,5.53);
\draw (-1.36,2.79)-- (-5.12,2.78);
\draw (-3.38,2.78) -- (-3.24,2.96);
\draw (-3.38,2.78) -- (-3.24,2.61);
\draw (3.75,6.95)-- (2.75,6.96);
\draw (3.12,6.95) -- (3.25,7.13);
\draw (3.12,6.95) -- (3.25,6.78);
\draw (1.32,2.78)-- (0.82,3.64);
\draw (1.01,3.33) -- (1.22,3.3);
\draw (1.01,3.33) -- (0.92,3.12);
\draw (5.15,2.76)-- (5.65,3.63);
\draw (5.47,3.31) -- (5.55,3.11);
\draw (5.47,3.31) -- (5.25,3.28);
\draw (7.83,2.79)-- (7.33,3.66);
\draw (7.51,3.34) -- (7.73,3.31);
\draw (7.51,3.34) -- (7.43,3.14);
\draw (11.66,2.8)-- (12.16,3.67);
\draw (11.97,3.35) -- (12.06,3.15);
\draw (11.97,3.35) -- (11.76,3.32);
\draw (10.23,6.98)-- (9.23,6.98);
\draw (9.6,6.98) -- (9.73,7.15);
\draw (9.6,6.98) -- (9.73,6.8);
\draw (1.82,3.64)-- (0.82,3.64);
\draw (1.19,3.64) -- (1.33,3.82);
\draw (1.19,3.64) -- (1.32,3.47);
\draw (4.65,3.63)-- (5.15,2.76);
\draw (4.97,3.08) -- (4.75,3.11);
\draw (4.97,3.08) -- (5.05,3.28);
\draw (4.65,3.63)-- (5.65,3.63);
\draw (5.29,3.63) -- (5.15,3.46);
\draw (5.29,3.63) -- (5.16,3.8);
\draw (3.74,4.16)-- (3.24,5.03);
\draw (3.43,4.72) -- (3.65,4.68);
\draw (3.43,4.72) -- (3.34,4.51);
\draw (2.74,4.17)-- (3.24,5.03);
\draw (3.06,4.72) -- (3.14,4.51);
\draw (3.06,4.72) -- (2.84,4.69);
\draw (3.74,4.16)-- (2.74,4.17);
\draw (3.11,4.16) -- (3.24,4.34);
\draw (3.11,4.16) -- (3.24,3.99);
\draw (3.25,6.09)-- (2.75,6.96);
\draw (2.93,6.64) -- (3.15,6.61);
\draw (2.93,6.64) -- (2.85,6.43);
\draw (9.24,4.19)-- (9.74,5.06);
\draw (9.55,4.74) -- (9.64,4.54);
\draw (9.55,4.74) -- (9.34,4.71);
\draw (10.24,4.19)-- (9.74,5.06);
\draw (9.92,4.74) -- (10.14,4.71);
\draw (9.92,4.74) -- (9.84,4.54);
\draw (10.24,4.19)-- (9.24,4.19);
\draw (9.6,4.19) -- (9.74,4.36);
\draw (9.6,4.19) -- (9.74,4.01);
\draw (7.33,3.66)-- (8.33,3.66);
\draw (7.96,3.66) -- (7.83,3.48);
\draw (7.96,3.66) -- (7.83,3.83);
\draw (7.83,2.79)-- (8.33,3.66);
\draw (8.14,3.34) -- (8.23,3.14);
\draw (8.14,3.34) -- (7.93,3.31);
\draw (11.66,2.8)-- (11.16,3.67);
\draw (11.34,3.35) -- (11.56,3.32);
\draw (11.34,3.35) -- (11.26,3.15);
\draw (11.16,3.67)-- (12.16,3.67);
\draw (11.79,3.67) -- (11.66,3.49);
\draw (11.79,3.67) -- (11.66,3.84);
\draw (9.23,6.98)-- (9.73,6.11);
\draw (9.55,6.43) -- (9.33,6.46);
\draw (9.55,6.43) -- (9.63,6.63);
\draw (10.23,6.98)-- (9.73,6.11);
\draw (9.92,6.43) -- (9.83,6.63);
\draw (9.92,6.43) -- (10.13,6.46);
\draw (0.82,3.64)-- (2.75,6.96);
\draw (1.86,5.42) -- (1.94,5.21);
\draw (1.86,5.42) -- (1.64,5.39);
\draw (5.65,3.63)-- (3.75,6.95);
\draw (4.64,5.41) -- (4.85,5.38);
\draw (4.64,5.41) -- (4.55,5.2);
\draw (2.74,4.17)-- (1.82,3.64);
\draw (2.16,3.84) -- (2.2,4.05);
\draw (2.16,3.84) -- (2.37,3.75);
\draw (5.15,2.76)-- (1.32,2.78);
\draw (3.1,2.77) -- (3.24,2.94);
\draw (3.1,2.77) -- (3.24,2.6);
\draw (11.66,2.8)-- (7.83,2.79);
\draw (9.61,2.8) -- (9.74,2.97);
\draw (9.61,2.8) -- (9.74,2.62);
\draw (8.33,3.66)-- (9.24,4.19);
\draw (8.9,3.99) -- (8.87,3.77);
\draw (8.9,3.99) -- (8.69,4.07);
\draw (11.16,3.67)-- (10.24,4.19);
\draw (10.58,4) -- (10.79,4.08);
\draw (10.58,4) -- (10.61,3.78);
\draw (-4.62,3.65)-- (-3.75,4.15);
\draw (-4.07,3.97) -- (-4.1,3.75);
\draw (-4.07,3.97) -- (-4.27,4.05);
\draw (0.51,2.01)-- (-0.49,2.01);
\draw (-0.13,2.01) -- (0.01,2.19);
\draw (-0.13,2.01) -- (0.01,1.84);
\draw (0.01,1.15)-- (-0.49,2.01);
\draw (-0.31,1.7) -- (-0.09,1.67);
\draw (-0.31,1.7) -- (-0.39,1.49);
\draw (0.01,1.15)-- (0.51,2.01);
\draw (0.33,1.7) -- (0.41,1.49);
\draw (0.33,1.7) -- (0.11,1.67);
\draw (-2.37,-1.24)-- (-1.86,-2.11);
\draw (-2.05,-1.79) -- (-2.27,-1.76);
\draw (-2.05,-1.79) -- (-1.96,-1.59);
\draw (-2.37,-1.24)-- (-1.37,-1.24);
\draw (-1.73,-1.24) -- (-1.87,-1.42);
\draw (-1.73,-1.24) -- (-1.87,-1.07);
\draw (-1.37,-1.24)-- (-1.86,-2.11);
\draw (-1.68,-1.79) -- (-1.77,-1.59);
\draw (-1.68,-1.79) -- (-1.46,-1.76);
\draw (2.39,-1.24)-- (1.89,-2.11);
\draw (2.07,-1.79) -- (1.99,-1.59);
\draw (2.07,-1.79) -- (2.29,-1.76);
\draw (1.39,-1.24)-- (1.89,-2.11);
\draw (1.71,-1.79) -- (1.49,-1.76);
\draw (1.71,-1.79) -- (1.79,-1.59);
\draw (0.01,0.13)-- (-0.49,-0.74);
\draw (-0.31,-0.42) -- (-0.39,-0.21);
\draw (-0.31,-0.42) -- (-0.09,-0.39);
\draw (0.51,-0.74)-- (-0.49,-0.74);
\draw (-0.12,-0.74) -- (0.01,-0.56);
\draw (-0.12,-0.74) -- (0.01,-0.91);
\draw (0.01,0.13)-- (0.51,-0.74);
\draw (0.33,-0.42) -- (0.11,-0.39);
\draw (0.33,-0.42) -- (0.41,-0.21);
\draw (0.51,2.01)-- (2.39,-1.24);
\draw (1.52,0.27) -- (1.3,0.3);
\draw (1.52,0.27) -- (1.6,0.47);
\draw (2.39,-1.24)-- (1.39,-1.24);
\draw (1.76,-1.24) -- (1.89,-1.07);
\draw (1.76,-1.24) -- (1.89,-1.41);
\draw (7.33,1.92)-- (6.33,1.92);
\draw (6.69,1.92) -- (6.83,2.1);
\draw (6.69,1.92) -- (6.83,1.75);
\draw (6.83,1.06)-- (7.33,1.92);
\draw (7.15,1.61) -- (7.23,1.4);
\draw (7.15,1.61) -- (6.93,1.58);
\draw (4.45,-1.33)-- (4.96,-2.2);
\draw (4.77,-1.88) -- (4.55,-1.85);
\draw (4.77,-1.88) -- (4.86,-1.68);
\draw (4.45,-1.33)-- (5.45,-1.33);
\draw (5.09,-1.33) -- (4.95,-1.51);
\draw (5.09,-1.33) -- (4.95,-1.16);
\draw (5.45,-1.33)-- (4.96,-2.2);
\draw (5.14,-1.88) -- (5.05,-1.68);
\draw (5.14,-1.88) -- (5.36,-1.85);
\draw (9.21,-1.33)-- (8.71,-2.2);
\draw (8.89,-1.88) -- (8.81,-1.68);
\draw (8.89,-1.88) -- (9.11,-1.85);
\draw (6.83,0.04)-- (6.33,-0.83);
\draw (6.51,-0.51) -- (6.43,-0.3);
\draw (6.51,-0.51) -- (6.73,-0.48);
\draw (7.33,-0.83)-- (6.33,-0.83);
\draw (6.7,-0.83) -- (6.83,-0.65);
\draw (6.7,-0.83) -- (6.83,-1);
\draw (6.83,0.04)-- (7.33,-0.83);
\draw (7.15,-0.51) -- (6.93,-0.48);
\draw (7.15,-0.51) -- (7.23,-0.3);
\draw (6.33,1.92)-- (4.45,-1.33);
\draw (5.32,0.18) -- (5.24,0.38);
\draw (5.32,0.18) -- (5.54,0.21);
\draw (7.33,1.92)-- (9.21,-1.33);
\draw (8.34,0.18) -- (8.12,0.21);
\draw (8.34,0.18) -- (8.42,0.38);
\draw (6.83,1.06)-- (6.83,0.04);
\draw (6.83,0.41) -- (6.66,0.55);
\draw (6.83,0.41) -- (7,0.55);
\draw (9.21,-1.33)-- (8.21,-1.33);
\draw (8.58,-1.33) -- (8.71,-1.16);
\draw (8.58,-1.33) -- (8.71,-1.5);
\draw (5.45,-1.33)-- (6.33,-0.83);
\draw (6.01,-1.01) -- (5.98,-1.23);
\draw (6.01,-1.01) -- (5.81,-0.93);
\draw (-1.36,2.79)-- (-5.12,2.78);
\draw (-5.62,3.65)-- (-3.75,6.9);
\draw (-4.62,5.39) -- (-4.53,5.19);
\draw (-4.62,5.39) -- (-4.83,5.36);
\draw (-0.86,3.65)-- (-2.75,6.9);
\draw (-1.87,5.4) -- (-1.65,5.37);
\draw (-1.87,5.4) -- (-1.96,5.19);
\draw (4.65,3.63)-- (3.74,4.16);
\draw (4.08,3.97) -- (4.29,4.05);
\draw (4.08,3.97) -- (4.11,3.75);
\draw (-0.86,3.65)-- (-1.86,3.64);
\draw (-1.5,3.64) -- (-1.36,3.82);
\draw (-1.5,3.64) -- (-1.36,3.47);
\draw (-1.86,3.64)-- (-2.74,4.16);
\draw (-2.42,3.97) -- (-2.21,4.05);
\draw (-2.42,3.97) -- (-2.39,3.75);
\draw (9.23,6.98)-- (7.33,3.66);
\draw (8.21,5.2) -- (8.13,5.4);
\draw (8.21,5.2) -- (8.43,5.23);
\draw (12.16,3.67)-- (10.23,6.98);
\draw (11.13,5.44) -- (11.34,5.41);
\draw (11.13,5.44) -- (11.04,5.24);
\draw (-1.86,-2.11)-- (1.89,-2.11);
\draw (0.15,-2.11) -- (0.01,-2.28);
\draw (0.15,-2.11) -- (0.01,-1.93);
\draw (-2.37,-1.24)-- (-0.49,2.01);
\draw (-1.36,0.5) -- (-1.28,0.3);
\draw (-1.36,0.5) -- (-1.58,0.47);
\draw (0.01,0.13)-- (0.01,1.15);
\draw (0.01,0.78) -- (0.18,0.64);
\draw (0.01,0.78) -- (-0.16,0.64);
\draw (-0.49,-0.74)-- (-1.37,-1.24);
\draw (-1.04,-1.06) -- (-1.01,-0.84);
\draw (-1.04,-1.06) -- (-0.84,-1.14);
\draw (0.51,-0.74)-- (1.39,-1.24);
\draw (1.07,-1.06) -- (0.87,-1.14);
\draw (1.07,-1.06) -- (1.04,-0.84);
\draw (4.96,-2.2)-- (8.71,-2.2);
\draw (6.97,-2.2) -- (6.83,-2.37);
\draw (6.97,-2.2) -- (6.83,-2.02);
\draw (7.33,-0.83)-- (8.21,-1.33);
\draw (7.89,-1.15) -- (7.69,-1.23);
\draw (7.89,-1.15) -- (7.86,-0.93);
\draw (9.74,5.06)-- (9.73,6.11);
\draw (9.73,5.72) -- (9.91,5.59);
\draw (9.73,5.72) -- (9.56,5.59);
\draw (3.1,5.6) node[anchor=north west] {$z$};
\draw (9.55,5.6) node[anchor=north west] {$w$};
\draw (-0.09,0.84) node[anchor=north west] {$t$};
\draw (6.69,0.66) node[anchor=north west] {$u$};
\draw (1.32,2.78)-- (1.82,3.64);
\draw (1.64,3.33) -- (1.72,3.12);
\draw (1.64,3.33) -- (1.42,3.3);
\draw (3.75,6.95)-- (3.25,6.09);
\draw (3.43,6.4) -- (3.35,6.61);
\draw (3.43,6.4) -- (3.65,6.43);
\draw (3.25,6.09)-- (3.24,5.03);
\draw (3.25,5.43) -- (3.07,5.56);
\draw (3.25,5.43) -- (3.42,5.56);
\draw (6.33,1.92)-- (6.83,1.06);
\draw (6.65,1.37) -- (6.43,1.4);
\draw (6.65,1.37) -- (6.73,1.58);
\draw (8.71,-2.2)-- (8.21,-1.33);
\draw (8.39,-1.65) -- (8.61,-1.68);
\draw (8.39,-1.65) -- (8.31,-1.85);
\begin{scriptsize}
\draw[color=black] (-4.09,5.01) node {A};
\draw[color=black] (-2.41,5.01) node {A};
\draw[color=black] (-3.29,3.42) node {B};
\draw[color=black] (-1.82,6.2) node {C};
\draw[color=black] (2.25,5) node {F};
\draw[color=black] (4.19,5) node {E};
\draw[color=black] (3.15,3.44) node {H};
\draw[color=black] (8.88,5) node {F};
\draw[color=black] (10.55,5) node {E};
\draw[color=black] (4.78,6.12) node {D};
\draw[color=black] (9.68,3.42) node {C};
\draw[color=black] (11.28,6.15) node {H};
\draw[color=black] (-3.27,2.49) node {c};
\draw[color=black] (7.36,3.08) node {$e_1$};
\draw[color=black] (2.59,4.55) node {$e_2$};
\draw[color=black] (8.03,3.83) node {$e_2$};
\draw[color=black] (1.48,5.61) node {a};
\draw[color=black] (4.96,5.56) node {a};
\draw[color=black] (2.12,4.27) node {a};
\draw[color=black] (3.23,2.44) node {c};
\draw[color=black] (9.7,2.54) node {c};
\draw[color=black] (8.62,4.27) node {a};
\draw[color=black] (10.84,4.24) node {a};
\draw[color=black] (-4.35,4.24) node {b};
\draw[color=black] (1.32,-1.92) node {$e_3$};
\draw[color=black] (1.79,0.63) node {b};
\draw[color=black] (5.09,0.55) node {c};
\draw[color=black] (8.52,0.63) node {a};
\draw[color=black] (6.58,0.63) node {e};
\draw[color=black] (5.66,-0.71) node {c};
\draw[color=black] (-5.07,5.53) node {a};
\draw[color=black] (-1.51,5.53) node {a};
\draw[color=black] (4.39,4.24) node {d};
\draw[color=black] (-2.03,4.19) node {b};
\draw[color=black] (7.95,5.61) node {a};
\draw[color=black] (11.48,5.66) node {d};
\draw[color=black] (0.11,-2.44) node {b};
\draw[color=black] (-1.69,0.68) node {c};
\draw[color=black] (-0.25,0.79) node {e};
\draw[color=black] (-1.13,-0.66) node {c};
\draw[color=black] (1.17,-0.66) node {a};
\draw[color=black] (6.84,-2.51) node {a};
\draw[color=black] (8.05,-0.81) node {b};
\draw[color=black] (-0.02,-1.41) node {G};
\draw[color=black] (-0.97,0.09) node {I};
\draw[color=black] (0.91,0.04) node {J};
\draw[color=black] (1.4,1.61) node {B};
\draw[color=black] (5.73,-0.17) node {I};
\draw[color=black] (7.9,-0.14) node {J};
\draw[color=black] (8.31,1.64) node {D};
\draw[color=black] (6.79,-1.46) node {G};
\draw[color=black] (9.52,5.84) node {a};
\draw[color=black] (2.07,3.03) node {$e_1$};
\draw[color=black] (3.08,5.7) node {a};
\draw[color=black] (8.1,-2) node {$e_3$};
\end{scriptsize}
\end{tikzpicture}
\caption{Knot in $L(3,1)$. All edges are oriented, but for simplicity, only the necessary names have been given. \label{knot in ltroisun}}
\end{center}
\end{figure}
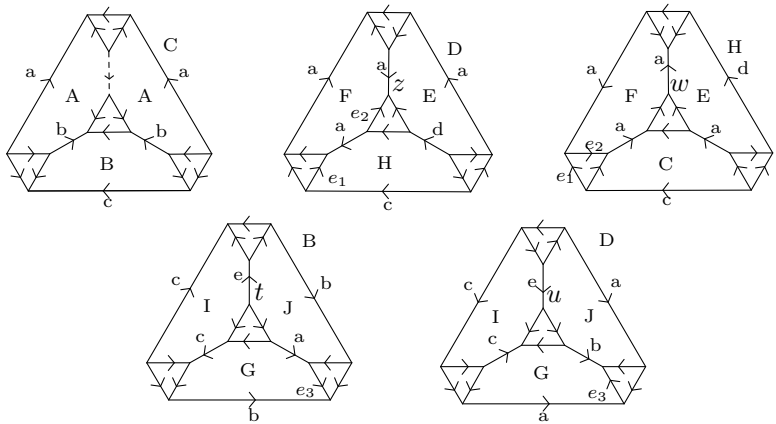

The figure \ref{knot in ltroisun} represents a truncated particular H-triangulation of a knot in $L(3,1)$. The knot is represented by the dotted edge. The shape parameters $z,t,u,w$ are actually associated to edges of the corresponding ideal triangulation. Let us remind that $H_1(L(3,1),\mathbb{Z})=\mathbb{Z}/3\mathbb{Z}$ and $H_2(L(3,1),\mathbb{Z})=0$.

The gluing equations are $t=u^{-1}$ and $z=w^{-1}$ and the presentation of the ring \Rh\ reduces to $\mathbb{Z}\langle u_1,u_{\bar{1}},u_2,u_3, u_{\bar{3}} \mid u_1u_{\bar{1}}=1,  u_3u_{\bar{3}}=1, u_2^3=-1 \rangle$.
$f$ is not surjective, because $\text{Im}(f)=\mathbb{Z}\langle u_1u_2^{2}, u_{\bar{1}}u_2,u_3, u_{\bar{3}}  \mid u_1u_{\bar{1}}=1,  u_3u_{\bar{3}}=1, u_2^3=-1 \rangle$.

\section{References}

\bibliographystyle{plain}
\bibliography{A_non_commutative_generalisation}

\Address

\end{document}